\documentclass{amsart}

\usepackage[colorlinks=true, linkcolor=blue]{hyperref}

\usepackage{amsmath,amsthm,amssymb}
\usepackage[margin=2.5cm]{geometry}
\usepackage{enumitem}
\usepackage{tikz-cd}
\usepackage{mathrsfs}

\DeclareFontFamily{U}{wncy}{}
\DeclareFontShape{U}{wncy}{m}{n}{<->wncyr10}{}
\DeclareSymbolFont{mcy}{U}{wncy}{m}{n}
\DeclareMathSymbol{\Sha}{\mathord}{mcy}{"58} 

\def\bs{\boldsymbol}

\setlist[itemize]{label=--}

\newtheorem{theorem}{Theorem}[section]
\newtheorem{lemma}[theorem]{Lemma}
\newtheorem{prop}[theorem]{Proposition}
\newtheorem{cor}[theorem]{Corollary}

\theoremstyle{definition}
\newtheorem{definition}[theorem]{Definition}

\newtheorem{hyp}[theorem]{Hypothesis}

\newtheorem*{def*}{Definition}
\newtheorem*{ex*}{Example}
\theoremstyle{remark}
\newtheorem{remark}[theorem]{Remark}
\newtheorem*{remark*}{Remark}

\newcommand{\A}{\mathbb{A}}

\newcommand{\bra}{\langle}
\newcommand{\C}{\mathbb{C}}

\newcommand{\cont}{\mathrm{cont}}

\newcommand{\et}{\mathrm{\acute{e}t}}

\newcommand{\Fr}{\mathrm{Frob}}
\newcommand{\G}{\mathbf{G}}
\newcommand{\GL}{\mathrm{GL}}

\newcommand{\h}{\mathrm{H}}

\newcommand{\iso}{\xrightarrow{\,\sim\,}}
\newcommand{\Iw}{\mathrm{Iw}}
\newcommand{\ket}{\rangle}
\newcommand{\m}{\mathfrak{m}}

\newcommand{\mf}[1]{\mathfrak{#1}}
\newcommand{\N}{\mathbb{N}}
\newcommand{\Norm}{\mathbf{N}}

\newcommand{\Or}{\mathcal{O}}
\newcommand{\p}{\mathfrak{p}}

\newcommand{\q}{\mathfrak{q}}
\newcommand{\Q}{\mathbb{Q}}
\newcommand{\R}{\mathbb{R}}

\newcommand{\sh}[1]{\mathcal{#1}}
\newcommand{\SL}{\mathrm{SL}}

\newcommand{\U}{\mathrm{U}}

\newcommand{\Z}{\mathbb{Z}}

\DeclareMathOperator{\Art}{Art}
\DeclareMathOperator{\Aut}{Aut}

\DeclareMathOperator{\diag}{diag}

\DeclareMathOperator{\End}{End}

\DeclareMathOperator{\Gal}{Gal}
\DeclareMathOperator{\Hom}{Hom}
\DeclareMathOperator{\Ind}{Ind}

\DeclareMathOperator{\ord}{ord}

\DeclareMathOperator{\Res}{Res}

\DeclareMathOperator{\spec}{Spec}

\DeclareMathOperator{\Stab}{Stab}

\DeclareMathOperator{\Tr}{Tr}

\theoremstyle{definition}
\newtheorem{assumption}[theorem]{Assumption}

\def\bs{\boldsymbol}

\newcommand{\bH}{\mathbf{H}}

\newcommand{\Sh}{\mathrm{Sh}}

\newcommand{\indf}{\mathbf{1}}
\newcommand{\tG}{{\widetilde{G}}}
\newcommand{\tX}{{\widetilde{X}}}
\newcommand{\con}[1]{{#1}^{\mathtt{c}}}

\begin{document}

\title{Anti-cyclotomic Euler system of diagonal cycles}
\author{Shilin Lai} 
\address[S.L.]{
The University of Texas at Austin, 2515 Speedway, PMA 8.100, Austin, TX 78712, USA.} 
\email{shilin.lai@math.utexas.edu}

\author{Christopher Skinner}
\address[C.S.]{Princeton University, Fine Hall, Washington Road, Princeton, NJ 08544-1000, USA}
\email{cmcls@princeton.edu}

\begin{abstract}
  We construct split anti-cyclotomic Euler systems for Galois representations attached to certain RACSDC automorphic representations on the group $\mathrm{GL}_n\times\mathrm{GL}_{n+1}$. As a result, we make progress towards certain rank 1 cases of the Beilinson--Bloch--Kato conjecture for those representations.
\end{abstract}

\maketitle
\setcounter{tocdepth}{1}
\tableofcontents

\section{Introduction}
Let $E/F$ be a CM extension. For $m=n,n+1$, let $\Pi_m$ be a RACSDC\footnote{regular, algebraic, conjugate self dual, cuspidal} automorphic representation on $\GL_m(\A_E)$. Let $p$ be a rational prime. By the work of many people (cf.~\cite{ChenevierHarris, Caraiani}), there is an $m$-dimensional geometric Galois representation $\rho_{\Pi_m}:\Gal_E\to\GL(V_{\Pi_m})$ attached to $\Pi_m$ 
such that
\[
  L(V_{\Pi_m},s) = L\bigg(s+\frac{1-m}{2},\Pi_m\bigg).
\]
Let $\Pi=\Pi_n\boxtimes\Pi_{n+1}$ and $V_{\Pi} = (V_{\Pi_n}\otimes V_{\Pi_{n+1}})(n)$. Then $V_\Pi$ is conjugate self-dual, and $L(V_\Pi,s) = L(s+1/2,\Pi)$.

We will work in the framework of the arithmetic Gan--Gross--Prasad (AGGP) conjecture \cite{GGPConjecture, ZhangAFL1}. This imposes the following two assumptions on $\Pi$:
\begin{itemize}
  \item The weights of $\Pi_n$ and $\Pi_{n+1}$ are perfectly interlacing at each archimedean place (Definition~\ref{def:Interlace});
  \item The root number $\varepsilon\big(\frac{1}{2},\Pi\big)$ is $-1$.
\end{itemize}
In this setting, the AGGP conjecture predicts that a certain natural diagonal cycle has a null-homologous modification $\triangle_\mathrm{GGP}$, and it satisfies a relation
\[
  L'\Big(\frac{1}{2},\Pi\Big)\overset{??}{=}(\ast)\bra\triangle_\mathrm{GGP},\triangle_\mathrm{GGP}\ket_{\mathrm{BB}},
\]
where $\bra\cdot,\cdot\ket_{\mathrm{BB}}$ is the Beilinson--Bloch height pairing and $(\ast)\neq 0$. Combined with the Bloch--Kato conjecture, this would imply the equivalence
\[
  \bra\triangle_\mathrm{GGP},\triangle_\mathrm{GGP}\ket_{\mathrm{BB}}\neq 0\overset{??}{\iff} \dim \h^1_f(E,V_\Pi)=1.
\]
When $n=1$ and $\Pi_2$ is a modular form of weight 2, the cycle $\triangle_\mathrm{GGP}$ is essentially a Heegner point, and a lot is known about the above two conjectures.

In this paper, we will consider the forward direction of the above conjectural equivalence. The left hand side, including the null-homologous cycle $\triangle_{\mathrm{GGP}}$, depends on the standard conjectures on cycles, so we instead consider its $p$-adic \'etale realization. This leads to an element $z_{\mathrm{GGP}}\in\h^1_f(E,V_\Pi)$, and we expect
\begin{equation}\label{eq:BK}
  z_{\mathrm{GGP}}\neq 0\overset{??}{\implies}\dim \h^1_f(E,V_\Pi)=1.\tag{$\dagger$}
\end{equation}
In the elliptic curves case, results of this form were first established by Kolyvagin \cite{Kolyvagin88} using what is now known as the method of Euler systems. In a recent work of Jetchev--Nekov\'{a}\v{r}--Skinner \cite{JNS}, an analogue of Kolyvagin's argument is axiomatized. The goal of this paper is to supply the Euler system which feeds into their formalism,
thereby establishing the expectation \eqref{eq:BK} under appropriate hypotheses on $\Pi$ and $V_\Pi$.

\subsection{Main results}
Let $\mathscr{L}$ be the set of places of $F$ that split in $E$, with finitely many places removed including all places dividing $p$ or at which $V_\Pi$ is ramified. For each $\ell\in\mathscr{L}$, fix a place $\lambda$ of $E$ above it. Let $\mathscr{R}^p$ be the set of square-free products of places in $\mathscr{L}$.

For each ideal $\m$ of $\Or_F$, let $E[\m]$ be the ring class field with conductor $\m$, so it is associated to the order $\Or_F+\m\Or_E$ by class field theory. Our first result is the following.

\begin{theorem}[Tame norm relation, Theorem~\ref{thm:ES}]\label{thm:tame-intro}
  Assuming the conjectural description of the cohomology of certain unitary Shimura varieties (Hypothesis~\ref{conj:Coh}), there is a lattice $T_\Pi$ in $V_\Pi$ and elements $c_\m\in\h^1(E[\m],T_\Pi)$ for $\m\in\mathscr{R}^p$ satisfying the norm relation
  \[
    \Tr_{E[\m]}^{E[\m\ell]}c_{\m\ell}=P_\lambda(\Fr_\lambda)c_\m
  \]
  where $P_\lambda(X)=\det(1-X\Fr_\lambda|V_\Pi)^{-1}$, and $\Fr_\lambda$ is the arithmetic Frobenius.
\end{theorem}

Up to this point, we have made no assumption on the prime $p$. Now suppose there is a place $\p$ of $F$ above $p$ such that
\begin{itemize}
  \item $\p$ splits in $E$,
  \item $\Pi$ is ordinary at $\p$, cf.~Definition~\ref{defn:ord}, which can be phrased as a condition on the $p$-adic valuation of the Satake parameters of $\Pi$.
\end{itemize}
then we also get norm-compatible classes above $\p$.
\begin{theorem}[Wild norm relation, Corollary~\ref{cor:ES}]
  For each $\m\in\mathscr{R}^p$, there exists classes $c_{\p^t\m}\in\h^1(E[\p^t\m],T_\Pi)$ such that
  \[
    \Tr_{E[\p^t\m]}^{E[\p^{t+1}\m]}c_{\p^{t+1}\m}=c_{\p^t\m}
  \]
  for all $t\geq 0$.
\end{theorem}

\begin{remark}
  The wild norm relation, at least in the minimal weight case, is already present in the works \cite{LoefflerSpherical} and \cite{LiuRS}. Our construction is essentially a repackaged version of theirs.
\end{remark}

Combined with the results of \cite{JNS}, we obtain the following result towards the Bloch--Kato conjecture.
\begin{theorem}[Theorem~\ref{thm:App}]
Assume Hypothesis~\ref{conj:Coh}. Let $z_E = \Tr_E^{E[1]} c_1 \in \h^1_f(E,V_\Pi)$.  Suppose
\begin{itemize} 
\item[\rm (i)] $V_\Pi$ is absolutely irreducible,
\item[\rm (ii)] there exists $\sigma \in \Gal_E$ that fixes $E[1](\mu_{p^\infty})$ and such that $\dim_\Phi V_\Pi/(\sigma-1)V_\Pi = 1$, 
\item[\rm (iii)] either 
\subitem{\rm (a)} there exists a place $\p$ of $F$ above $p$ which splits in $E$ such that $\Pi$ is ordinary at $\mathfrak{p}$, or
\subitem{\rm (b)} there exists $\gamma \in\Gal_E$ such that $\gamma$ fixes $E[1](\mu_{p^\infty})$
and $V_\Pi/(\gamma-1)V_\Pi = 0$.
\end{itemize}
Then
\[
z_E\neq 0 \implies \dim_\Phi\h^1_f(E,V_\Pi) = 1.
\]
\end{theorem}

The base class $z_E$ in the theorem is the $p$-adic \'etale realization of the diagonal cycle $\triangle_\mathrm{GGP}$ described above. In the case $n=1$ and $\Pi_2$ has weight 2, this is the classical result that if the Heegner point has infinite order, then the elliptic curve has rank 1 \cite{Kolyvagin88}.

We make a few remarks about the relation of this theorem with other recent results in the literature.
\begin{enumerate}
  \item In the minimal weight case, the theorem was also established in \cite[Theorem 1.1.8]{LTXZZ} with a different set of conditions on the prime $p$ (denoted $\ell$ there). Our hypotheses appear to be milder. In the notations of Definition 8.1.1 of \emph{op.~cit.},
  \begin{itemize}
    \item Conditions (L1), (L2), (L4), (L6), (L7) are not needed. In particular, we allow small or ramified $p$ and make no assumptions on the residual representations.
    \item Condition (L3) is our condition (i).
    \item Condition (L5) is analogous to our condition (ii), though they are not directly comparable.
    \item Our condition (iii) does not appear in their list. It can often be proven under some big image assumptions (cf.\ the proof of Corollary~\ref{cor:App}).
  \end{itemize}
  As we will explain in Corollary~\ref{cor:App} and the remark following it, under certain purely automorphic assumptions, the conclusion of the theorem holds for \emph{all} rational primes $p$.
  \item Also in the minimal weight case, a recent result of Disegni--Zhang \cite{DZ} relates the $p$-adic height pairing of the class $z_E$ to the derivative of a $p$-adic $L$-function. Using \cite[Theorem 1.1.8]{LTXZZ}, they deduce the following cases of the $p$-adic Bloch--Kato conjecture \cite[Theorem C]{DZ}:
  \[
    \ord_{\chi=\mathbf{1}}\mathscr{L}_p(\chi)=1  \implies \dim_\Phi\h^1_f(E,V_\Pi) = 1,
  \]
  under certain conditions on $p$, including the ordinarity of $\Pi$ at $p$ and the admissibility conditions discussed above.
  
  By applying Theorem~\ref{thm:App} instead, we may replace the condition ``admissible prime'' in \cite[Theorem C]{DZ} with our conditions (i) and (ii) (condition (iii)(a) was already assumed in \emph{op.~cit.~}to construct the $p$-adic $L$-function).
  \item In the ordinary setting, one can also formulate a Perrin-Riou type Iwasawa main conjecture. Under the above admissibility assumption, one divisibility of such a conjecture was established by \cite[Theorem 1.2.3(2)]{LTX}. The results of \cite{JNS} gives an alternative proof under conditions (i), (ii), (iii)(a) along with additional conditions on $T_\Pi$.
\end{enumerate}

\subsection{Idea of proof}
Our construction is an integral version of the zeta integral strategy first developed in the papers \cite{LSZ,LSZ-U3}. We outline the strategy in a loose way. One first constructs a map of the form
\[
  C_c^\infty(G,\Q_p)\text{ ``$\to$'' } \h^1(E,V_\Pi)
\]
by taking Hecke translates of the diagonal cycle class. By construction, it is Hecke equivariant and invariant under a subgroup $H(\A)$. Now write down some well-chosen test functions on the source. General Hecke-cyclicity results (Proposition~\ref{prop:Cyc}) imply the norm relation holds for some Hecke operator, and a zeta integral calculation (Proposition~\ref{prop:TameZeta}) identifies the Hecke operator as the one required for an Euler system.

One subtle issue is that in most settings, the map does not preserve the natural integral structures, and one needs to formulate the correct integral structure on the source, cf.~\cite[Definition 3.2.1]{LSZ-U3}. In this paper, we take $H$-coinvariants immediately and constructs a map
\[
  C_c^\infty(X,\Z_p)\text{ ``$\to$'' }\h^1(E,T_\Pi)
\]
where $X=H\backslash G$ is the spherical variety underlying the setting, and $T_\Pi$ is a fixed lattice in $V_\Pi$ (cf.\ Proposition~\ref{prop:cycle}). This provides a natural explanation of the integrality requirements appearing in previous works. Moreover, all of the norm relations are proven directly in $\h^1(E,T_\Pi)$, so we bypass some technical Iwasawa-theoretic arguments used in previous works, for example in the proof of \cite[Theorem 10.5.4(b)]{LSZ}. Thus, we can completely remove all conditions on $p$ when dealing with tame norm relations.

We also systematically use the augmented group $\tG=G\times\U(1)$, which was used in the work of Graham--Shah on Euler systems in the Friedberg--Jacquet setting \cite{GrahamShah}. Our use of the augmentation is generally cosmetic, but it allows us to fix the level of the Shimura variety for $G$ in the steps where we construct Galois cohomology classes from cycles. This somewhat simplifies the exposition.

\subsection{Future work}
To relate the Euler system to classical $L$-values (instead of derivatives of $p$-adic $L$-functions), one would hope for an explicit reciprocity law relating the diagonal cycle class $z_{\mathrm{GGP}}$ to a certain $p$-adic $L$-function outside of the range of interpolation, as in the classical work \cite{BDP}. In our case, it is likely that one has to consider variations of $\Pi$ in a Hida family, instead of just character twists.

In a different direction, we have reduced the Euler system norm relation to certain relations in the function space $C_c^\infty(X,\C)$. This becomes a question of unramified harmonic analysis on spherical varieties, and it is possible to avoid much of the ad hoc calculations of Section~\ref{sec:Tame} using this point of view. In a future joint work of the first named author with Li Cai and Yangyu Fan, we will explain this approach in the twisted Friedberg--Jacquet setting.

\subsection*{Acknowledgments}
We thank Ashay Burungale, Li Cai, Yangyu Fan, Dimitar Jetchev, and Wei Zhang for their interest and helpful conversations.

The second-named author was supported by the Simons Investigator Grant \#376203 from the Simons Foundation and
 the National Science Foundation Grant DMS-1901985.

\section{General set-up}
\subsection{Fields}
Let $p$ be a rational prime. Fix an isomorphism $\iota:\C\simeq\overline{\Q}_p$.

Let $F$ be a totally real field of degree $g$ with algebraic closure $\bar{F}$. Let $E$ be a CM extension of $F$ contained in $\bar{F}$. We fix the following choices
\begin{itemize}
  \item an embedding $\iota_\infty:\bar{F}\to\C$,
  \item a CM type $\Sigma^+_\infty$ for $E$ compatible with $\iota_\infty$.
\end{itemize}
This singles out an archimedean place of $F$, which we denote by $\infty_*$. Moreover, the CM type $\Sigma^+_\infty$ distinguishes one of the two extensions of $\infty_*$ to $E$. The complex conjugation on $\C$ induces 
via $\iota_\infty$
an element $\mathtt{c}\in\Gal_F$ which is non-trivial in $\Gal(E/F)$. 

Let $\A$ denote the adeles of $F$ and $\A_E$ denote the adeles of $E$. We will use the following convention: a superscript consisting of places indicates omitting those places, and a subscript consisting of places indicates only considering those places. For example, $\A_{E,f}=\A_E^\infty$ is the ring of finite adeles of $\sh{K}$, and $\A_f^p=\A^{p\infty}$ is the ring of finite adeles of $F$ omitting the places above $p$.

The Artin map of global class field theory will be denoted by
\[
  \Art_F:F^\times\backslash\A^\times\to\Gal(F^\mathrm{ab}/F),
\]
normalized so that a uniformizer is sent to a geometric Frobenius element. 
Similarly,
let $\Art_E$ denote the Artin map for $E$. Given an ideal $\mf{c}\subseteq\Or_F$, let $E[\mf{c}]$ denote the ring class field of $E$ associated with the order $\Or_{\mf{c}}:=\Or_F+\mf{c}\Or_E$. 

\subsection{Function spaces}
Let $X$ be a td-space.\footnote{locally compact, Hausdorff, and totally disconnected topological space} Given a ring $R$, let $C_c^\infty(X,R)$ denote the space of $R$-valued locally constant and compactly supported functions on $X$. This is an inductive limit of free $R$-modules. If $S\subseteq X$, then $\indf[S]$ will denote the indicator function of $S$.

Suppose $X$ has the form $H\backslash G$, where $H$ is a closed subgroup of a td-group $G$, then $X$ carries a right action of $G$. Suppose $R$ is a field of characteristic 0. Fix $R$-valued Haar measures $\mu_H$ and $\mu_G$ on $H$ and $G$ respectively. Let $\mu_X$ be their quotient measure on $X$. This defines a natural map
\[
  I_H:C_c^\infty(G,R)\to C_c^\infty(X,R),\quad \phi\mapsto\left(x\mapsto\int_{H} \phi(h\tilde{x})dh\right).
\]
where $\tilde{x}$ is any lift of $x\in X$ to $G$. If $U\subseteq G$ is open compact, then by direct computation,
\begin{equation}\label{eq:IH}
  I_H(\mathbf{1}[gU])=\mu_H(H\cap gUg^{-1})\mathbf{1}[HgU].
\end{equation}
This expression shows that $I_H$ is surjective. Therefore, $I_H$ identifies $C_c^\infty(X,R)$ with the $H$-coinvariants of $C_c^\infty(G,R)$, where $H$ acts by left translation.

Since both $G$ and $X$ have right $G$-actions, the above function spaces are left $G$-representations. In particular, note that for $g\in G$ and $U$ an open subset of either $G$ or $X$, this action is
\[
  g\cdot\indf[U]=\indf[Ug^{-1}].
\]
With the choice of $\mu_G$, the function spaces also carry left $C_c^\infty(G,R)$-module structures. In particular, if $\phi\in C_c^\infty(X,R)$ and $f\in C_c^\infty(G,R)$, then we have
\[
  (f\cdot\phi)(x)=\int_{G}\phi(xg)f(g)dg.
\]
The map $I_H$ is left $G$-equivariant, so it is also a map of left $C_c^\infty(G,R)$-modules.

Let $K'\subseteq K\subseteq G$ be two open compact subgroups. Define the trace map by
\[
  \Tr_K^{K'}:C_c^\infty(X,R)^{K'}\to C_c^\infty(X,R)^{K},\ \phi\mapsto\sum_{\gamma\in K/K'}\gamma\cdot\phi.
\]
It is equivalent to the action of the Hecke operator $\mu_G(K')^{-1}\indf[K]$, but this definition makes it visibly independent of the choice of Haar measures.
\begin{lemma}\label{lem:Trace}
  Let $K'\subseteq K,K''$ be three open compact subgroups of $G$. Let $x\in X$, then
  \[
    \Tr_K^{K'}(\indf[xK''])=\indf[xK]\cdot[K'':K']\cdot[\Stab_{K}(x):\Stab_{K''}(x)].
  \]
\end{lemma}
\begin{proof}
  Let $\tilde{x}\in G$ denote a lift of $x$, then we have the relation
  \[
    \Tr_K^{K'}(\indf[\tilde{x}K''])=\indf[\tilde{x}K]\cdot[K'':K']
  \]
  in $C_c^\infty(G,\Q)$. The coinvariant map $I_H$ is $G$-equivariant, so applying to to both sides gives
  \[
    \Tr_K^{K'}(\indf[xK''])\cdot\mu_H(H\cap\tilde{x}K''\tilde{x}^{-1})=\indf[xK]\cdot[K'':K']\mu_H(H\cap\tilde{x}K\tilde{x}^{-1}).
  \]
  This is a relation in $C_c^\infty(X,\Q)$. Dividing both sides by the Haar measure term gives
  \[
    \Tr_K^{K'}(\indf[xK''])=\indf[xK]\cdot[K'':K']\cdot\frac{\mu_H(H\cap\tilde{x}K\tilde{x}^{-1})}{\mu_H(H\cap\tilde{x}K''\tilde{x}^{-1})}.
  \]
  It remains to observe that $H\cap\tilde{x}K\tilde{x}^{-1}=\tilde{x}\Stab_K(x)\tilde{x}^{-1}$, so the quotient of volumes is equal to the (generalized) index $[\Stab_K(x):\Stab_{K''}(x)]$.
\end{proof}

\subsection{Groups and representations}
\subsubsection{Hermitian spaces}\label{ss:Herm}
Let $\mathtt{V}_n$ be a Hermitian space of dimension $n$ defined with respect to $E/F$. Fix an orthogonal basis $(e_1,\cdots,e_n)$ for $\mathtt{V}_n$ with the following properties:
\begin{itemize}
  \item $\bra e_1,e_1\ket\in F$ is negative in the distinguished embedding $\infty_*$, and positive in the other embeddings.
  \item If $i>1$, then $\bra e_i,e_i\ket$ is totally positive.
\end{itemize}
Here $\bra -,-\ket$ is the implicit Hermitian pairing on $\mathtt{V}_n$.
This forces $\mathtt{V}_n$ to have signature $(n-1,1)$ at $\infty_*$ and $(n,0)$ at the other archimedean places. In other words,
$\mathtt{V}_n$ is standard indefinite in the sense of \cite[Definition 3.2.1]{LTXZZ}.

Let $\mathtt{V}_{n+1}=\mathtt{V}_n\oplus Ee_{n+1}$ equipped with a Hermitian pairing extending $\bra -,-\ket$ such that $e_{n+1}$ is a vector with a totally positive norm and $(e_1,\cdots,e_n,e_{n+1})$ is an orthogonal basis. 
Then $\mathtt{V}_{n+1}$ is also standard indefinite.

\subsubsection{Unitary groups}
Let $H=\U(\mathtt{V}_n)$ be the unitary group of $\mathtt{V}_n$, namely the algebraic group over $F$ consisting of the isometries of $\mathtt{V}_n$. If $\ell$ is a place of $F$ which splits in $E$, then the choice of a place $\lambda$ of $E$ above $\ell$ gives an identification
\[
  H\times_F F_\ell\simeq\GL_n(F_\ell)
\]
using the basis $(e_1,\cdots,e_n)$. This also applies to the other groups introduced below.

Let $G=\U(\mathtt{V}_n)\times\U(\mathtt{V}_{n+1})$ be the product of the associated unitary groups. The containment $\mathtt{V}_n\subseteq\mathtt{V}_{n+1}$ gives rise to an embedding $\U(\mathtt{V}_n)\hookrightarrow\U(\mathtt{V}_{n+1})$. Let $\triangle:H\hookrightarrow G$ be the graph of this embedding.

It is useful to introduce an augmented group. Let $\mathtt{V}_1=\det(\mathtt{V}_n)$, a 1-dimensional Hermitian space. The unitary group $\U(\mathtt{V}_1)$ is the norm one subgroup of $\Res_{E/F}\G_m$. Form the group
\[
  \widetilde{G}=\U(\mathtt{V}_n)\times\U(\mathtt{V}_{n+1})\times\U(\mathtt{V}_1),
\]
with the embedding $G\hookrightarrow\widetilde{G}$ sending $(g_n,g_{n+1})$ to $(g_n,g_{n+1},\det(g_n))$. Let $\widetilde{\triangle}$ denote the composite embedding $H\hookrightarrow\tG$. Explicitly, $\widetilde{\triangle}(h)=(\triangle(h),\det(h))$. We will typically treat $H$ as a subgroup of $G$ or $\widetilde{G}$ and not mention $\triangle$ or $\widetilde{\triangle}$ explicitly.

Let $X=H\backslash G$ as an affine algebraic variety, with the right $G$-action by right translation. For any $F$-algebra $R$, we have $X(R)=H(R)\backslash G(R)$, and it can be identified with $\U(\mathtt{V}_{n+1})(R)$ via the map $u\mapsto (1,u)\in G$.
Let $\widetilde{X}=H\backslash\widetilde{G}$, then we get an induced closed embedding $X\hookrightarrow\widetilde{X}$.

\subsubsection{Integral representations}\label{ss:IntCoeff}
We now fix notations for lattices in algebraic representations of $p$-adic groups. The discussion of this subsection applies to any reductive group over $F$.

Let $\bH=\Res_{F/\Q}H$ be the algebraic group over $\Q$ obtained by restriction of scalar, so
\[
  \bH(\Q_p)=\prod_{\q|p}H(F_\q).
\]
Let $P$ be a fixed minimal parabolic of $\bH$, and let $A$ be a maximally split torus of $\bH$ contained in $P$. This defines a set $\Delta^+$ of positive roots of $A$. Let
\[
  A^-=\{a\in A\,|\,v_p(\alpha(a))\leq 0\text{ for all }\alpha\in\Delta^+\},
\]
and similarly define $A^+\subseteq A$ using the condition $v_p(\alpha(a))\geq 0$.

Let $K_H\subseteq\bH(\Q_p)$ be a subgroup which has an Iwahori decomposition $K_H=\overline{N}_K\cdot M_K\cdot N_K$ with respect to $P$, where the notations are set-up so that for all $a\in A^-$,
\[
  a^{-1}N_K a\subseteq N_K,\quad a\overline{N}_K a^{-1}\subseteq\overline{N}_K,
\]
and moreover $M_K$ is contained in the centralizer of $A$ in $P$. By \cite[Proposition 1.4.4]{CasselmanBook}, every neighbourhood of the identity contains such a group $K_H$.\footnote{Our $A^+$ is Casselman's $A^-$.} If $a,b\in A^-$, then
\begin{align}
  (K_H a K_H)\cdot(K_H b K_H)&=K_H a(\overline{N}_K M_K N_K) b K_H \notag \\
  &=K_H(a\overline{N}_K a^{-1})M_K(b^{-1}N_K b)K_H \label{eq:2} \\
  &\subseteq K_H ab K_H. \notag
\end{align}
The reverse inclusion clearly holds, so it is an equality. Let
\begin{equation}\label{eq:H-}
  H_p^-=\bigcup_{a\in A^-}K_H a K_H.
\end{equation}
The above calculations show that this is an open submonoid of $\bH(\Q_p)$. Define $H_p^+$ analogously, so $H_p^+$ and $H_p^-$ are mutual inverses in $\bH(\Q_p)$.

Let $\Phi$ be a finite extension of $\Q_p$ so that $\bH$ splits over $\Phi$. We will view $\Phi$ as a subfield of $\bar{\Q}_p$. Let $\Or$ be its ring of integers, and let $t$ be a uniformizer. Finally, let $\nu_\Phi$ denote the additive valuation on $\Phi$ normalized so that $\nu_\Phi(t)=1$. Fix a positive Weyl chamber of $\bH\times_{\Q_p}\Phi$ containing $\Delta^+$. Given a dominant character $\mu$ of $\bH$, write its associated highest weight representation as
\[
  \xi_\mu:\bH(\Q_p)\to\GL(V_\mu).
\]
Further fix an $\Or$-lattice $\Lambda_\mu\subseteq V_\mu$ which is invariant under $\xi_\mu(K_H)$. 
\begin{lemma}\label{lem:IntRepn}
  If $a\in A^-$, then $t^{-\nu_\Phi(\mu(a))}\xi_\mu(a)\in\End_\Or(\Lambda_\mu)$. Moreover, this extends to a monoid homomorphism
  \begin{equation}
    \xi_\mu^\circ:H_p^-\to\End_\Or(\Lambda_\mu).
  \end{equation}
\end{lemma}
\begin{proof}
  The extension part is clear in view of equation~\eqref{eq:2} and the comment following it. For integrality, note that for a vector $v$ of weight $\mu'$, we have
  \[
    t^{-\nu_\Phi(\mu(a))}\xi_\mu(a)v=t^{-\nu_\Phi(\mu(a))}\mu'(a)v=t^{\nu_\Phi(\mu'\mu^{-1}(a))}\cdot(\text{unit})\cdot v.
  \]
  Since $a\in A^-$ and $\mu$ is the highest weight, the exponent of $t$ is non-negative, as required.
\end{proof}

\subsubsection{Branching law}\label{ss:Branching}
In the standard pinning for $\GL_n$, the dominant characters can be labelled by non-increasing tuples $\underline{a}\in\Z^n$ by the formula
\[
  \mu_{\underline{a}}:\diag(t_1,\cdots,t_n)\mapsto \prod_{i=1}^n t_i^{a_i}.
\]
Denote the corresponding highest weight representation by $(\xi_{\underline{a}},V_{\underline{a}})$. Define $-\underline{a}=(-a_n,-a_{n-1},\cdots,-a_1)$, so $V_{-\underline{a}}=V_{\underline{a}}^\vee$. 

Using the basis $(e_1,\cdots,e_n)$ for the Hermitian space, we have an identification
\[
  \bH\times_\Q\bar{\Q}\simeq\prod_{\Sigma_\infty^+\ni\tau:E\hookrightarrow\bar{\Q}}\GL_n,
\]
so algebraic representations of $\bH$ are external tensor products of $V_{\underline{a}}$ for various choices of $\underline{a}$. We write the dominant characters for $\bH$ as a $g$-tuple $\underline{\bs{a}}=(\underline{a}^{\tau})$, where $\tau$ ranges over all embeddings $E\hookrightarrow\bar{\Q}$ in the fixed CM type $\Sigma_\infty^+$. The same discussion applies to the group $\G=\Res_{F/\Q}G$, so its dominant characters is written as a pair
\[
  (\underline{\bs{a}},\underline{\bs{b}})\in \Z^{n\times g}\times\Z^{(n+1)\times g}.
\]
Applying the classical branching law one embedding at a time gives the following.

\begin{prop}\label{prop:Branching}
  If $\underline{\bs{a}}\in\Z^{n\times g}$ and $\underline{\bs{b}}\in\Z^{(n+1)\times g}$ satisfy the inequalities 
  \[
    b_1^\tau\geq a_1^\tau\geq b_2^\tau\geq\cdots\geq a_n^\tau\geq b_{n+1}^\tau
  \]
  for all $\tau:E\hookrightarrow\bar{\Q}$ in $\Sigma_\infty^+$,
  then the space
  \[
    \Hom_{\bH}(1,V_{-\underline{\bs{a}}}\boxtimes V_{\underline{\bs{b}}})
  \]
  is one-dimensional.
\end{prop}

\begin{definition}\label{def:Interlace}
  An algebraic representation $\xi_\mu$ with highest weight $\mu=(\underline{\bs{a}},\underline{\bs{b}})$ is \emph{perfectly interlacing} if the condition of the above proposition holds.
\end{definition}

In this case, the lattice $\Lambda_\mu$ fixed in the previous subsection contains an $\bH$-fixed vector $v_*$ that is non-vanishing in $\Lambda_\mu/t\Lambda_\mu$. Such a $v_*$ is unique up to $\Or^\times$-multiple. Define the branching map by
\[
  \mathtt{br}:\Or\to\Lambda_\mu,\quad 1\mapsto v_*.
\]
Choose the open compact subgroups $K_H$ and $K_G$ to be compatible under the diagonal embedding. Then for $h\in H_p^-$, its image $\triangle(h)$ lies in $G_p^-$. Let $H_p^-$ act on $\Or$ via the monoid homomorphism
\[
  H_p^-\to (\Or,\times),\quad h\mapsto t^{-(\nu_\Phi\circ\mu\circ\triangle)(h)}.
\]
By construction, this makes the branching morphism $H_p^-$-equivariant. This additional constant factor will not be relevant to us, since we will never consider Hecke operators from $H$ except in the proof of Proposition~\ref{prop:Cyc}.

Finally, we extend a representation of $G$ to one for $\tG$ by making the augmentation factor act trivially. The branching law and branching morphism goes through without any change.

\subsection{Shimura varieties}
\subsubsection{Shimura data}
Define an algebraic homomorphism
\[
  h_n:\Res_{\C/\R}\G_m\to\Res_{F/\Q}\U(\mathtt{V}_n)\times_\Q\R,\quad z\mapsto \left(\big(\begin{smallmatrix}\con{z}/z & \\ & \mathrm{id}_{n-1}\end{smallmatrix}\big)_{\infty_*},\mathrm{id}_n,\cdots,\mathrm{id}_n\right).
\]
Let $X_n$ be the conjugacy class of $h_n$, then the pair $(\U(\mathtt{V}_n),X_n)$ is a Shimura datum of abelian type with reflex field $E$, viewed as a subfield of $\C$ by $\iota_\infty$. This is the Shimura datum used in \cite{LTXZZ} and conjugate to the one in \cite{GGPConjecture}.

If $U\subseteq\U(\mathtt{V}_n)(\A_f)$ is a neat open compact subgroup, then denote the associated Shimura variety by $\Sh_n^\circ(U)$. It is a smooth algebraic variety over $E$ of dimension $n-1$. As $U$ varies, we obtain an inverse system of varieties with finite \'etale transition maps. This system will be denoted by $\Sh_n^\circ$. Let $g\in\U(\mathtt{V}_n)(\A_f)$, then we have a right translation map
\[
  T_g:\Sh_n^\circ(K)\to\Sh_n^\circ(g^{-1}Kg)
\]
satisfying $T_{gh}=T_h\circ T_g$. This gives a right action of $\U(\mathtt{V}_n)(\A_f)$ on the inverse system $\Sh_n^\circ$.

\begin{remark}\label{remark:neat}
  In what follows, we will sometimes implicitly assume the level structure is neat. This can be arranged by choosing a sufficiently small open compact subgroup at finitely many places and requiring our level to be contained in them \cite[Remark 1.4.1.9]{Lan-PEL}. This does not impact the discussion.
\end{remark}

The Shimura data $(\U(\mathtt{V}_{n+1}),X_{n+1})$ and $(\U(\mathtt{V}_1),X_1)$ can be defined similarly. By taking products, we get Shimura data $(G,X_G)$ and $(\widetilde{G},X_{\widetilde{G}})$. Let $(H,X_H)=(\U(\mathtt{V}_n),X_n)$, then the embeddings defined above induces embeddings of Shimura data
\[
  (H,X_H)\to (G,X_G)\to (\widetilde{G},X_{\widetilde{G}}).
\]
We adopt similar notations for Shimura varieties. Their reflex fields are all contained in $E$, so we view them as varieties over $E$. If $K_{\tG}\subseteq \tG(\A_f)$ and $K_H\subseteq H(\A_f)$ are both open compact, and $K_H\subseteq K_\tG\cap H(\A_f)$, then we get a morphism
\begin{equation}\label{eq:embedding}
  \iota_{K_H,K_\tG}:\Sh_H^\circ(K_H)\to\Sh_{\tG}^\circ(K_\tG),
\end{equation}
which is still defined over $E$. Moreover, for each $K_H$, there exists a $K_\tG$ such that this morphism is a closed embedding, by \cite[Proposition 1.15]{Deligne-Bourbaki}. We summarize these properties by saying there is a closed embedding of inverse systems $\Sh_H^\circ\hookrightarrow\Sh_{\tG}^\circ$.

\subsubsection{Extensions of the base field}\label{ss:Extension}
We now consider the Shimura set $\Sh_1$ more carefully. By Hilbert's theorem 90, we have the following isomorphism
\[
  E^\times\A_F^\times\backslash\A_E^\times/E_\infty^\times\iso \Sh_1(\C)=E^{\times,1}\backslash \A_{E,f}^{\times,1},\quad s\mapsto\frac{\con{s}}{s},
\]
where the superscript 1 means the norm one elements. Given an ideal $\mf{c}$ of $F$, recall that $\Or_{\mf{c}}=\Or_F+\mf{c}\Or_E$. Assume that $\mf{c}$ is coprime to 2 times the discriminant ideal of $E/F$, then under the above isomorphism, $\hat{\Or}_{\mf c}^\times$ gets mapped to the open compact subgroup
\[
  J_{\mf{c}}=\{x\in\A_{E,f}^{\times,1}\,|\,x\equiv 1\pmod{\mf{c}}\}\subseteq\U(\mathtt{V}_1)(\A_f).
\]
It follows that we have an isomorphism.
\[
  E^\times\backslash\A_E^\times/\hat{\Or}_{\mf{c}}^\times E_\infty^\times\iso \Sh_1(J_{\mf c})(\C)=E^{\times,1}\backslash \A_{E,f}^{\times,1}/J_{\mf{c}}.
\]
By class field theory, the left hand side can be identified with $\spec E[\mf{c}]$ as schemes over $E$. The resulting action of $\Aut(\C/E)$ agrees with the one given by the Shimura reciprocity law on the right hand side. Therefore, we obtain an isomorphism of schemes over $E$
\[
  \Sh_1(J_{\mf c})\simeq\spec E[\mf{c}].
\]
Consequently,
\begin{equation}\label{eq:Gal1}
  \h^0_\et(\Sh_1(J_{\mf c})_{/\bar{E}},\Z_p)=\Z_p[\Sh_1(J_{\mf c})]=\Ind_E^{E[\mf{c}]}\Z_p.
\end{equation}
Moreover, if $K\subseteq G(\A_f)$ is an open compact subgroup, then
\begin{equation}\label{eq:Sh1}
  \Sh_{\widetilde{G}}(K\times J_{\mf c})\simeq \Sh_G(K)\times_E E[\mf{c}].
\end{equation}
This will be used to interpret cycles on $\Sh_{\widetilde{G}}$ as cycles on $\Sh_G$ defined over field extensions of $E$.

The Hecke algebra for $\tG$ decomposes as
\[
  C_c^\infty(\tG(\A_f),\Z)\simeq C_c^\infty(G(\A_f),\Z)\otimes_\Z C_c^\infty(\U(\mathtt{V}_1)(\A_f),\Z).
\]
Under the isomorphism~\eqref{eq:Sh1}, the action of a Hecke operator $\sh{T}\otimes\indf[xJ_{\mf{c}}]$ on the left hand side becomes $\sh{T}\otimes\Art_E(s)$ on the right hand side, where $s\in\A_E^\times$ is any element such that $x=\con{s}/s$.

\subsubsection{Coefficients}
We follow the formalism in \cite[\S III.2]{HTLanglands} to construct $p$-adic \'etale local systems on Shimura varieties attached to the lattices considered in Subsection~\ref{ss:IntCoeff}. The discussion only assumes $H$ is a reductive group over $F$ with a Shimura variety. In particular, it also applies to $\tG$.

Fix an open compact subgroup of $\bH(\Q_p)=\prod_{\q|p}H(F_\q)$ with an Iwahori factorization. Using it, define
\[
  H(\A_f)^\pm:=H(\A_f^p)\times H_p^\pm,
\]
where the notations are as in Subsection~\ref{ss:IntCoeff}. Let $\xi^\circ:H_p^-\to\End_\Or(\Lambda)$ be one of the monoid representations constructed in Lemma~\ref{lem:IntRepn}. By letting $H(\A_f^p)$ act trivially, it can be extended to a monoid homomorphism $H(\A_f)^-\to\End_\Or(\Lambda)$, which we continue to denote by $\xi^\circ$.

Let $K\subseteq H(\A_f)$ be an open compact subgroup contained in $H(\A_f)^-$. The inverse system of varieties $\{\Sh_H(K')\}_{K'\subseteq K}$ forms a Galois covering of $\Sh_H(K)$ whose Galois group is $K$. Let $\mathbb{L}_\xi^{(K)}$ be the local system corresponding to the restriction $\xi^\circ|_{K}$. More precisely, it is defined by
\[
  \mathbb{L}_\xi^{(K)}=\varprojlim_n\mathbb{L}_{\xi,n}^{(K)},
\]
where each $\mathbb{L}_{\xi,n}^{(K)}$ is the $p^n$-torsion local system corresponds to the torsion representation on $\Lambda/p^n\Lambda$. We will often drop the superscript if the level is clear. If $g\in H(\A_f)^-$ and $K,K'$ are two level structures as above such that $g^{-1} K g\subseteq K'$, then we have morphisms
\[
  T_g:\Sh_H(K)\to\Sh_H(K'),\quad \vec{T}_g:T_g^*\mathbb{L}_{\xi}^{(K')}\to\mathbb{L}_{\xi}^{(K)}.
\]
They satisfies the usual cocycle condition $\vec{T}_{gh}=\vec{T}_g\circ T_g^*\vec{T}_h$. We may pass to the infinite level and define
\begin{equation}\label{eq:CohDefn}
  \h^\bullet_\et(\Sh_{H/\bar{E}},\mathbb{L}_\xi):=\varinjlim_{K\subseteq H(\A_f)^-}\h^\bullet_\et\big(\Sh_H(K)_{/\bar{E}},\mathbb{L}_\xi^{(K)}\big),
\end{equation}
where the transition maps are the pullbacks induced by $(T_g,\vec{T}_g)$. Moreover, write
\[
  \h_\cont^\bullet\big(\Sh_H(K),\mathbb{L}_\xi^{(K)}\big)
\]
for the continuous \'etale cohomology of the inverse system $(\mathbb{L}_\xi^{(K)}/p^n\mathbb{L}_\xi^{(K)})_{n\geq 1}$ defined in \cite{JannsenContEt}. We can then define
\begin{equation}\label{eq:ContDefn}
  \h_\cont^\bullet(\Sh_H,\mathbb{L}_\xi):=\varinjlim_{K\subseteq H(\A_f)}\h_\cont^\bullet\big(\Sh_H(K),\mathbb{L}_\xi^{(K)}\big).
\end{equation}
Both $\h_\et^\bullet$ and $\h_\cont^\bullet$ carry natural left $H(\A_f)^-$-actions given by the pullback $T_g^*$.

All the constructions are naturally functorial when there is a morphism of representations $\xi_1^\circ\to\xi_2^\circ$. Moreover, if $\Lambda=\Z_p$ and $\xi^\circ$ is trivial, then the local systems are just the constant sheaf $\Z_p$, and we have transition morphisms for all $g\in H(\A_f)$.

\subsubsection{Cycle class map}
Let $\xi$ be a perfectly interlacing representation for $G$, which we extend trivially to $\tG$. The goal of this subsection is to construct a collection of classes in $\h_\cont^{2n}(\Sh_{\tG},\mathbb{L}_\xi(n))$ labelled by functions on the spherical variety $\tX(\A_f)$. Throughout this subsection, suppose the open compact subgroups $K_{H,p}$ and $K_{\tG,p}$ chosen for the lattice constructions are compatible under the diagonal embedding.

Let $g\in\tG(\A_f)^+$, and let $K\subseteq\tG(\A_f)$ be an open compact subgroup such that
\begin{equation}\label{eq:a}
  K, gKg^{-1}\subseteq\tG(\A_f)^+.
\end{equation}
We will define a class
\[
  cyc(g,K)\in \h^{2n}_\cont(\Sh_\tG(K),\mathbb{L}_\xi)
\]
as follows. Let $K_H^{(g)}=H(\A_f)\cap gKg^{-1}$, then we have two canonical maps of Shimura varieties with local systems
\begin{align*}
  &\iota_{g,K}:\Sh_H(K_H^{(g)})\to \Sh_{\tG}(gKg^{-1}),\quad \vec{\iota}_{g,K}:\iota_{g,K}^*\mathbb{L}_{\xi^\circ}\iso\mathbb{L}_{\xi|_H^\circ}\\
  &T_{g^{-1}}:\Sh_{\tG}(gKg^{-1})\leftarrow\Sh_{\tG}(K),\quad \vec{T}_{g^{-1}}:T_{g^{-1}}^*\mathbb{L}_{\xi}\to\mathbb{L}_\xi.
\end{align*}
This has an induced map on cohomology
\[
  \h^0_\cont(\Sh_H(K_H^{(g)}),\mathbb{L}_{g\xi|_H^\circ})\xrightarrow{\,\iota_{g,K,*}\,}\h^{2n}_\cont(\Sh_\tG(gKg^{-1},\mathbb{L}_{g\xi}(n))\xrightarrow{\,T_{g^{-1}}^*}\h^{2n}_\cont(\Sh_\tG(K),\mathbb{L}_\xi(n))
\]
which we temporarily denote by $\varphi_{g,K}$.
\begin{remark}
  It is not true in general that $\iota_{g,K}$ is a closed embedding, but it always factors as a closed embedding followed by an \'etale covering map, so we may still define the cycle class map $\iota_{g,K,*}$. For details on verifying that this is well-defined and other compatibility properties, see \cite[Appendix A.1]{GrahamShah}.
\end{remark}

In Subsection~\ref{ss:Branching}, we constructed a branching morphism $\mathtt{br}$, which gives rise to map
\[
  \mathtt{br}_{g,*}:\h^0_\cont(\Sh_H(K_H^{(g)}),\Z_p)\to\h^0_\cont(\Sh_H(K_H^{(g)}),\mathbb{L}_{\xi|_H^\circ}).
\]
The left hand side contains the fundamental class $[\Sh_H(K_H^{(g)})]$, as defined in \cite[Theorem 3.23]{JannsenContEt}. We define $cyc(g,K_{\tG})$ to be its image under $\varphi_{g,K}\circ\mathtt{br}_{g,*}$.

\begin{prop}\label{prop:cycle}
Let $\tX(\A_f)^+$ denote the image of $\tG(\A_f)^+$ in $\tX(\A_f)$. There exists a unique left $\tG(\A_f)^-$-equivariant map
\[
  \mathtt{cyc}:C_c^\infty(\tX(\A_f)^+,\Z_p)\to\h^{2n}_\cont(\Sh_\tG,\mathbb{L}_\xi(n))
\]
such that for any open compact subgroup $K\subseteq\tG(\A_f)^+$ and element $g\in\tG(\A_f)^+$ satisfying \eqref{eq:a},
\[
  \mathtt{cyc}(\indf[\bar{g}K])=cyc(g,K),
\]
where the bar on $g$ indicates taking the image in $\tX(\A_f)$.
\end{prop}

\begin{proof}
  As $(g,K)$ runs over the pairs appearing in the statement of the proposition, the sets $\bar{g}K$ form a basis of open sets in $X(\A_f)^+$. Uniqueness follows from this. For equivariance, suppose $g'\in\tG(\A_f)^+$, then
  \[
    (g')^{-1}\cdot\indf[\bar{g}K]=\indf\big[\bar{g}K\cdot g'\big]=\indf\big[\overline{gg'}((g')^{-1}Kg')\big].
  \]
  On the cycle side, having fixed $g$ and $g'$, we may suppose $K$ is sufficiently small so that all the following maps are defined. The map $\varphi_{gg',(g')^{-1}Kg'}$ is the induced map on cohomology from the following maps
  \[
  \begin{tikzcd}
    \Sh_H(K_H^{(g)})\rar["{\iota_{g,K}}"] & \Sh_\tG(g Kg^{-1}) & \Sh_\tG(K)\lar["{T_{g^{-1}}}"'] & \Sh_\tG((g')^{-1}Kg')\lar["{T_{(g')^{-1}}}"'].
  \end{tikzcd}
  \]
  Since $g^{-1},(g')^{-1}\in\tG(\A_f)^-$, the final two arrows compose to $T_{(gg')^{-1}}$. It follows that $cyc(gg',(g')^{-1}Kg')=T_{(g')^{-1}}^*cyc(g,K)$, as required.
  
  It remains to show that $\mathtt{cyc}$ is well-defined. This reduces to the following two claims.
  \begin{enumerate}
    \item If $h\in H(\A_f)\cap\tG(\A_f)^+$ and $K$ is sufficiently small, then $cyc(hg,K)=cyc(g,K)$.
    \item If $K_1\subseteq K_2$ are two open compact subgroups of $\tG(\A_f)^+$, and
    \[
      \bar{g} K_2=\bigsqcup_{\gamma\in I}\overline{g\gamma}K_1
    \]
    for some finite set $I$, then $cyc(g,K_2)=\sum_{\gamma\in I}cyc(g\gamma, K_1)$.
  \end{enumerate}
  Claim (1) follows by unwinding definitions, using in particular the $H_p^-$-invariance of $\mathtt{br}$. In claim (2), we may assume $K_1\trianglelefteq K_2$ by shrinking $K_1$, and then we may assume $g=1$ by performing a translation on $\Sh_\tG$. Moreover, by claim (1), we may use any set of representatives $I$ in the decomposition.
  
  The strategy now is to decompose the transition from $K_2$ to $K_1$ into two steps. Consider the embedding
  \[
    \frac{H(\A_f)\cap K_2}{H(\A_f)\cap K_1}\hookrightarrow\frac{K_2}{K_1}.
  \]
  Let $K$ be the subgroup of $K_2$ containing $K_1$ corresponding to the image of this embedding, then
  \[
    H(\A_f)\cap K=H(\A_f)\cap K_2,\quad [H(\A_f)\cap K:H(\A_f)\cap K_1]=[K:K_1].
  \]
  From these, we will show the following two equalities:
  \begin{align}
    \indf[\overline{K_2}]&=\sum_{\gamma\in K\backslash K_2}\indf[\overline{K\gamma}]\label{eq:5}\\
    &=\sum_{\gamma\in K\backslash K_2}\indf[\overline{K_1\gamma}].\label{eq:6}
  \end{align}
  For \eqref{eq:5}, suppose $hk\gamma=k'\gamma'$, where $k,k'\in K$, $h\in H(\A_f)$, and $\gamma,\gamma'\in K\backslash K_2$. Then $h=k'\gamma'\gamma^{-1}k^{-1}$ is in both $H(\A_f)$ and $K_2$, so it lands in $K$, which implies $\gamma'\in K\gamma$. For \eqref{eq:6}, Lemma~\ref{lem:Trace} gives the equality
  \[
    \Tr^{K_1}_K\indf[\overline{K_1}]=\indf[\overline{K}]\cdot[K:K_1].
  \]
  This can only happen if all $[K:K_1]$ terms of the left hand side agree, so $\overline{K_1}=\overline{K}$ as subsets of $X(\A_f)$. Translating by $\gamma$ proves \eqref{eq:6}.
  
  Recall we have reduced to the case where $K_1$ is a normal subgroup of $K_2$, so $K_1\gamma=\gamma K_1$ for all $\gamma\in K_2$. Therefore, the above calculations give a decomposition
  \[
    \bar{1}K_2=\bigsqcup_{\gamma\in K\backslash K_2}\bar{\gamma}K_1.
  \]
  Claim (2) reduces to checking that $cyc(1,K_2)=\sum_{\gamma\in K\backslash K_2}cyc(\gamma,K_1)$. Similar to before, we have two steps.
  
  The first step parallels equation~\eqref{eq:5} and goes from $K_2$ to $K$. By construction, we have the following commutative diagram
  \[
  \begin{tikzcd}
     & \Sh_\tG(K)\dar["{\mathrm{pr}_{K_1,K}}"]\\
    \Sh_H(H(\A_f)\cap K)\dar[equal]\rar["{\iota_{1,K}}"] & \Sh_\tG(K)\dar["{\mathrm{pr}_{K,K_2}}"]\\
    \Sh_H(H(\A_f)\cap K_2)\rar["{\iota_{1,K_2}}"] & \Sh_\tG(K_2).
  \end{tikzcd}
  \]
  The two vertical maps on the right are \'etale covering maps, and their composition $\mathrm{pr}_{K_1,K_2}$ is Galois. On the level of sheaves, we have the relation
  \[
    \mathrm{pr}_{K_1,K_2}^*\mathrm{pr}_{K,K_2,*}=\sum_{\delta\in K_2/K}\mathrm{pr}_{K_1,\delta K\delta^{-1}}^*T_\delta^*=\sum_{\gamma\in K\backslash K_2}\mathrm{pr}_{K_1,\gamma^{-1}K\gamma}^*T_{\gamma^{-1}}^*,
  \]
  which can be verified by working \'etale-locally, where it reduces to an elementary calculation about Galois sets. Let $[-]$ denote the cycle class of a variety, then
  \begin{align}
    \mathrm{pr}_{K_1,K_2}^*cyc(1,K_2)&=\mathrm{pr}_{K_1,K_2}^*\iota_{1,K_2,*}[\Sh_H(H(\A_f)\cap K_2)]\notag\\
    &=\mathrm{pr}_{K_1,K_2}^*\mathrm{pr}_{K,K_2,*}\iota_{1,K,*}[\Sh_H(H(\A_f)\cap K)]\label{eq:7}\\
    &=\sum_{\gamma\in K\backslash K_2}\mathrm{pr}_{K_1,\gamma^{-1}K\gamma}^* T_{\gamma^{-1}}^* cyc(1,K).\notag
  \end{align}
  
  The second step goes from $K$ to $K_1$. Let $\gamma\in K\backslash K_2$. Since $K_1\trianglelefteq K_2$, we have $\gamma^{-1}K\gamma\supseteq K_1$, and hence the following commutative diagram
  \[
  \begin{tikzcd}
    \Sh_H(H(\A_f)\cap K_1)\dar\rar["{\iota_{1,K_1}}"]& \Sh_\tG(K_1)\dar["{\mathrm{pr}_{K_1,K}}"]\rar["T_{\gamma}"] & \Sh_\tG(K_1)\dar["{\mathrm{pr}_{K_1,\gamma^{-1}K\gamma}}"]\\
    \Sh_H(H(\A_f)\cap K)\rar["{\iota_{1,K}}"] & \Sh_\tG(K)\rar["{T_{\gamma}}"] & \Sh_\tG(\gamma^{-1}K\gamma).
  \end{tikzcd}
  \]
  The left vertical edge is a covering map of degree $[K:K_1]$ by the choice of $K$, so the left square is Cartesian. It is explained in \cite[Appendix A.1]{GrahamShah} that the cycle class map in continuous \'etale cohomology is compatible with base change. Applying it to the left square, we get
  \[
    \mathrm{pr}_{K_1,K}^* cyc(1,K)=cyc(1,K_1).
  \]
  Using the right square, this implies
  \begin{equation}\label{eq:8}
    \mathrm{pr}_{K_1,\gamma^{-1}K\gamma}^* T_{\gamma^{-1}}^* cyc(1,K)=cyc(\gamma,K_1).
  \end{equation}
  Combining equations~\eqref{eq:7} and \eqref{eq:8} gives the desired relation $cyc(1,K_2)=\sum_{\gamma\in K\backslash K_2}cyc(\gamma,K_2)$.
\end{proof}

\begin{remark}
  Assuming $\xi$ is the trivial coefficient system for simplicity, then we have the following diagram
  \[
  \begin{tikzcd}
    C_c^\infty(\tG(\A_f),\Q_p)\dar["I_{H(\A_f)}"']\drar & \\
    C_c^\infty(\tX(\A_f),\Q_p)\rar["\mathtt{cyc}"] & \h^{2n}_\cont(\Sh_\tG,\Q_p(n)).
  \end{tikzcd}
  \]
  The diagonal arrow is the cycle class map considered in previous works, for example in \cite[\S 8.2]{LSZ}. It multiplies the cycles $cyc(gK)$ by certain volume terms to ensure compatibility, but this destroys integrality. This volume term is exactly the same as the one introduced in the coinvariant map $I_{H(\A_f)}$, so we immediately know that $\mathtt{cyc}$ preserves the integral structure. The above proof is a refinement of the proof in \emph{loc.~cit.} working directly at the level of integral coefficients.
\end{remark}

\subsubsection{Cohomology of unitary Shimura varieties}
We begin with the single unitary group $H$. Let $\pi_n$ be a cuspidal automoprhic representation of $H(\A)=\U(\mathtt{V}_n)(\A)$. Let $\pi_n^\vee$ be its contragredient representation. Let $\xi$ be an algebraic representation of $\Res_{F/\Q} H$.

\begin{definition}\label{defn:weight}
  We say that $\pi_n$ has \emph{weight} $\xi$ if its archimedean component is in the discrete series whose infinitesimal character agrees with that of $\xi$.
\end{definition}

The following definition captures the contribution of its finite part to the cohomology group of $\Sh_n$.
\begin{definition}\label{defn:V}
  Recall that $\iota$ was a fixed identification $\C\iso\bar{\Q}_p$. Define
  \[
    V^{n-1}(\pi_n;\xi):=\Hom_{\bar{\Q}_p[\U(\mathtt{V}_n)(\A_f)]}\big(\iota\pi_n^{\vee,\infty},\h^{n-1}_\et(\Sh_{n/\bar{E}},\mathbb{L}_\xi\otimes_{\Or}\bar{\Q}_p)\big).
  \]
\end{definition}
This a finite-dimensional $\bar{\Q}_p$-vector space with the 
continuous $\bar{\Q}_p$-linear action of $\Gal_E$ coming from the \'etale cohomology term. Its Galois structure is conjecturally described in \cite[\S 10]{KottwitzAA} (when $F\neq\Q$). Based on this, we formulate the following hypothesis.

\begin{hyp}\label{conj:Coh}
  \begin{enumerate}
    \item Up to possibly expanding $\Phi$, the finite part $\iota\pi_n^\infty$ can be defined over $\Phi$, and there is a Galois representation
    \[
      \rho_{\pi_n}:\Gal_E\to\GL_n(\Phi)
    \]
    attached to $\pi_n$, in a sense made precise in \cite[Proposition 3.2.4]{LTXZZ}. In particular, $\rho_{\pi_n}$ is conjugate self-dual up to a twist: $\con{\rho}_{\pi_n}\simeq \rho_{\pi_n}^\vee(1-n)$.
    \item If $\pi_n$ has weight $\xi$, then the semisimplication of $V^{n-1}(\pi_n;\xi)$ is a subrepresentation of $\rho_{\pi_n}$. 
    \item If $\pi_n$ is stable,\footnote{has a cuspidal base change to $\GL_n(\A_E)$} then $V^{n-1}(\pi_n;\xi)$ is isomorphic to $\rho_{\pi_n}$.
  \end{enumerate}
\end{hyp}
\begin{remark}
  \begin{enumerate}
    \item The algebraicity statement of part (1) is a consequence of the existence of base change to $\GL_n(\A_E)$ (expected to be established in the sequel to \cite{KMSW}) and a result of Clozel \cite{Clozel_RA}.
    \item If $F=\Q$, then the Shimura varieties are not proper, and Kottwitz's discussion only applies to the intersection cohomology of its minimal compactification. Nonetheless, we still expect this hypothesis since our representation should be tempered, and these are not expected to contribute to the boundary cohomology (cf.~footnote 9 in \cite[\S 7]{LiuRS}).
    \item In the case of the trivial coefficient, this is \cite[Hypothesis 3.2.9]{LTXZZ}, where the current status of this expectation is also discussed. In particular, when $F\neq\Q$ and $\pi$ is stable, the hypothesis with coefficients will be established in a sequel to \cite{KSZ}.
  \end{enumerate}
\end{remark}

Given any level structure $K_n\subseteq\U(\mathtt{V}_n)(\A_f)$, by the definition of $V^{n-1}(\pi_n;\xi)$, there is a projection map
\[
  \mathtt{pr}_n:\pi_{n,f}^{K_n}\otimes_{\Phi[K_n\backslash \U(\mathtt{V}_n)(\A_f)/K_n]}\h^{n-1}_\et(\Sh_n(K_n)_{/\bar{E}},\mathbb{V}_\xi)\to V^{n-1}(\pi_n;\xi),
\]
where $\pi_{n,f}$ is just another notation for $\pi_n^\infty$. 
We can also perform all the above constructions for $\U(\mathtt{V}_{n+1})(\A)$. 

Let $\pi=\pi_n\otimes\pi_{n+1}$ be a cuspidal automorphic representation of $G(\A)$ such that $\pi_{n}$ and $\pi_{n+1}$ are as above. Define the Galois representation
\[
  V_\pi:=V^{n-1}(\pi_n,\xi_n)\otimes V^n(\pi_{n+1},\xi_{n+1})(n).
\]
Let $K$ be a level structure for $G(\A_f)$, then by applying K\"unneth decomposition to the product $\mathtt{pr}_n\otimes\mathtt{pr}_{n+1}$, we get a map
\[
  \mathtt{pr}:\pi_f^K\otimes_{\Phi[K\backslash G(\A_f)/K]}\h^{2n-1}_\et(\Sh_G(K)_{/\bar{E}},\mathbb{L}_\xi(n))\to V_\pi.
\]
If $v\in\pi_f^K$, then $\mathtt{pr}_v$ will denote the induced projection map $\h^{2n-1}_\et(\Sh_G(K)_{/\bar{E}},\mathbb{V}_\xi(n))\to V_\pi$.

\begin{remark}
  The representation $V_\pi$ is pure of weight $-1$, and we will construct classes in $\h^1(E,V_\pi)$ satisfying certain norm relations. If $\pi$ is stable, then Hypothesis~\ref{conj:Coh} implies that the classes we constructed are related to the Bloch--Kato conjecture for the Galois representation attached to $\pi$.
\end{remark}

\subsubsection{Abel--Jacobi map}\label{ss:AJMap}
Fix an open compact subgroup $K_p\subseteq (\Res_{F/\Q}G)(\Q_p)$ so that $\pi^{K_p}$ is non-trivial. To trivialize our cycles, we choose in addition a neat tame level $K^p\subseteq G(\A_f^p)$ and a tame Hecke operator $\mathtt{t}\in C_c^\infty(K^p\backslash G(\A_f^p)/K^p,\Phi)$ such that
\begin{itemize}
  \item $\mathtt{t}$ annihilates $\h^{2n}_\et(\Sh_G(K^pK_{G,p})_{/\bar{E}},\mathbb{L}_\xi(n))$.
  \item The action of $\mathtt{t}$ on $\pi^{K^p}$ is given by a scalar $\omega_{\mathtt{t}}\in\Phi^\times$.
\end{itemize}
The existence of $K^p$ and $\mathtt{t}$ follows from \cite[Proposition 6.9(1)]{LiLiu-I}, whose proof still applies when coefficient systems are involved (cf.~the discussion in \cite[\S 7]{LiuRS} for the non-compact case). In the rest of this subsection, let $K=K^pK_{p}$. Let $U\subseteq\U(\mathtt{V}_1)(\A_f)$ be an arbitrary open compact subgroup, then $K\times U\subseteq\tG(\A_f)$ defines a level on $\tG$.

We now construct an integral version of the Abel--Jacobi map. Consider the Hochschild--Serre spectral sequence \cite[Remark 3.5(b)]{JannsenContEt}
\[
  E_2^{i,j}=\h^i\big(E,\h^j_\et(\Sh_{\tG}(K\times U)_{/\bar{E}},\mathbb{L}_\xi(n))\big)\Longrightarrow \h^{i+j}_\cont(\Sh_{\tG}(K\times U),\mathbb{L}_\xi(n)).
\]
There are no longer edge maps coming out of the $E_2^{i,j}$-term if $i\leq 1$. Therefore, $E_\infty^{0,2n}$ and $E_\infty^{1,2n-1}$ are submodules of their corresponding entries on page 2. From this, we get maps
\begin{align*}
  &\mathrm{cl}_0:\h^{2n}_\cont(\Sh_{\tG}(K\times U),\mathbb{L}_\xi(n))\to E_\infty^{0,2n}\hookrightarrow\h^{2n}_\et(\Sh_{\tG}(K\times U)_{/\bar{E}},\mathbb{L}_\xi(n)),\\
  &\mathrm{cl}_1:\ker(\mathrm{cl}_0)\to E_\infty^{1,2n-1}\hookrightarrow \h^1\big(E,\h^{2n-1}_\et(\Sh_{\tG}(K\times U)_{/\bar{E}},\mathbb{L}_\xi(n))\big).
\end{align*}

Suppose $\xi$ is a perfectly interlacing representation, so $\mathtt{cyc}$ is defined. Let $f\in C_c^\infty(\tX(\A_f)^+,\Z_p)^{K\times U}$. Then $\mathrm{cl}_0(\mathtt{t}\cdot\mathtt{cyc}(f))=0$ as $\mathtt{t}$ kills the entire degree $2n$ cohomology group for $\Sh_G(K)_{/\bar{E}}$. We can now define
\begin{align*}
  \mathrm{AJ}_\mathtt{t}:C_c^\infty(\tX(\A_f)^+,\Z_p)^{K\times U}&\to\h^1\big(E,d_\mathtt{t}^{-1}\h^{2n-1}_\et(\Sh_{\tG}(K\times U)_{/\bar{E}},\mathbb{L}_\xi(n))\big)\\
  f &\mapsto \omega_{\mathtt{t}}^{-1}\mathrm{cl}_1(\mathtt{t}\cdot\mathtt{cyc}(f)),
\end{align*}
where $d_\mathtt{t}$ is the denominator introduced from applying the Hecke operator $\mathtt{t}$ and dividing by $\omega_{\mathtt{t}}$. It is independent of $U$.

By the K\"unneth formula, 
\begin{align*}
  d_{\mathtt{t}}^{-1}\h^{2n-1}_\et(\Sh_{\tG}(K\times U)_{/\bar{E}},\mathbb{L}_\xi(n))&\simeq d_{\mathtt{t}}^{-1}\h^{2n-1}_\et(\Sh_G(K)_{/\bar{E}},\mathbb{L}_\xi(n))\otimes_\Or\h^0_\et(\Sh_1(U)_{/\bar{E}},\Or)\\
  &=d_{\mathtt{t}}^{-1}\h^{2n-1}_\et(\Sh_G(K)_{/\bar{E}},\mathbb{L}_\xi(n))\otimes_\Or \Or[\Sh_1(U)].
\end{align*}
Let $v\in\pi_f^K$. Then we can project the first term to $V_\pi$ using $\mathtt{pr}_v$. This defines a lattice $T_\pi\subseteq V_\pi$ which depends on both $v$ and $\mathtt{t}$, but not on $U$. All together, we have constructed a map
\[
  \mathrm{AJ}_{\mathtt{t},v}^{(U)}:C_c^\infty(\tX(\A_f)^+,\Z_p)^{K\times U}\to\h^1(E,T_\pi\otimes\Or[\Sh_1(U)]).
\]
If $\mathtt{cyc}(f)$ is already cohomologically trivial, then this agrees with applying $\mathtt{pr}_v$ to $\mathrm{cl}_1(\mathtt{cyc}(f))$. In this sense, the resulting cohomology class is a reasonable replacement for the conjectural canonical splitting of $\mathrm{cl}_0$.

Taking direct limits as $U$ varies over all level structures in $\U(\mathtt{V}_1)$ gives
\begin{equation}\label{eq:AJ}
  \mathrm{AJ}_{\mathtt{t},v}:C_c^\infty(\tX(\A_f)^+,\Z_p)^{K}\to\h^1(E,T_\pi\otimes\Or[\Sh_1]).
\end{equation}
This is equivariant under the action of Hecke operators coming from $\U(\mathtt{V}_1)$. Suppose $\sh{T}$ is a Hecke operator on $G(F_\ell)$ at a place $\ell\nmid p$ where $K^p$ is hyperspecial, and $\sh{T}v=\lambda v$. Then we find
\[
  \mathrm{AJ}_{\mathtt{t},v}(\sh{T}\cdot f)=\lambda\cdot \mathrm{AJ}_{\mathtt{t},v}(f)
\]
by chasing through the equivariance properties of $\mathtt{cyc}$ and $\mathtt{pr}$. This also holds when $\ell|p$, provided that $\sh{T}$ is supported on $G_p^-$, and we take into account the extra factor in the definition of the representation $\xi^\circ$.

\section{Wild norm relation}\label{sec:Wild}
\subsection{Split local notations}\label{ss:LocalNotations}
In this and the following sections, we will work locally at a place of $F$ which splits in $E$. To simplify notations, we will replace everything from before with local notations. These are in force in this section and the next.

Let $F$ be a local field with ring of integers $\Or$, uniformizer $\varpi$, and residue field $k_F$. Let $q=\# k_F$. Let $v:F^\times\to\Z$ be the additive valuation map, normalized so that $v(\varpi)=1$. Let $J_0=\Or^\times$. For each $t\geq 1$, let
\[
  J_t=\{x\in\Or^\times\,|\,x\equiv 1\pmod{\varpi^t}\}\subseteq F^\times.
\]

Let $m<n$ be two positive integers, define the embedding
\[
  \iota_{m,n}:\GL_m(F)\hookrightarrow\GL_n(F),\ g\mapsto\begin{pmatrix}
    g & \\ & \mathrm{id}_{n-m}
  \end{pmatrix}.
\]
If the dimensions 
$m$ and $n$ 
are clear, the subscripts will be dropped. Let $H=\GL_n(F)$ and $G=\GL_n(F)\times\GL_{n+1}(F)$. We have the diagonal embedding $\triangle:H\to G,\ h\mapsto(h,\iota_{n,n+1}(h))$, which will be implicitly used to view $H$ as a subgroup of $G$. Define the augmented group $\widetilde{G}=G\times F^\times$ with the embedding $G\hookrightarrow\widetilde{G},\ (g_n,g_{n+1})\mapsto(g_n,g_{n+1},\det g_n)$. Let $X=H\backslash G$ and $\tX=H\backslash\tG$. Then we have the induced closed embedding $X\hookrightarrow\tX$.

\subsection{Norm relation}\label{ss:WildRelation}
Let $B_\tG$ denote the upper triangular Borel subgroup of $\tG$, and let $K_G^\Iw\times\Or^\times$ denote the corresponding Iwahori subgroup of $\tG$. Let $\overline{B}_\tG$ denote the opposite Borel. The right action of $\overline{B}_\tG$ on $\tX$ has a unique open orbit. Let $\mathbf{x}\in \tX$ denote an element in the open orbit. Explicitly, we take it as the image of $\xi=(\mathrm{id}_n,\xi_{n+1},1)\in K_G^\Iw\times\Or^\times$, where
\[
  \xi_{n+1}=\begin{pmatrix}
    w_n & 1_{1\times n}\\
    0_{n\times 1} & 1
  \end{pmatrix}
\]
and $w_n$ is the matrix with 1s on the anti-diagonal and 0s everywhere else.\footnote{The algebraic representation $\xi$ will not make an appearance in this section.} 
The the stabilizer of $\mathbf{x}$ in $\tG$ is $\xi^{-1}H\xi$, which consists of elements of the form
\[
  \left(h,\begin{pmatrix}
    w_nhw_n & (w_nh-1)\cdot 1_{1\times n}\\
    0_{n\times 1} & 1
  \end{pmatrix},\det h\right)
\]
for $h\in H=\GL_n(F)$. Observe that this intersects trivially with $\overline{B}_\tG$, so the orbit map gives an open embedding
\begin{equation}\label{eq:OpenOrbit}
  \overline{B}_\tG\hookrightarrow\tX,\quad g\mapsto\mathbf{x}g,
\end{equation}
which we use to identify $\overline{B}_\tG$ with a subset of $\tX$. Incidentally, since both $\overline{B}_{\tG}$ and $\tX$ have the same dimension, this computation proves our earlier assertion that $\xi$ is a representative of an element in the open orbit.

Let $\tau=(\tau_n,\tau_{n+1})=\triangle(\diag(\varpi^n,\varpi^{n-1},\cdots\varpi))\in G$. For each $t\geq 0$, we can define the function
\[
  \delta^{(t)}=\indf[\mathbf{x}\tau^t(K_G^\Iw\times J_t)]\in C_c^\infty(\tX,\Z)^{K_G^\Iw\times J_t}.
\]
Define the Hecke operator
\[
  \sh{U}:C_c^\infty(\tX,\Z)^{K_G^\Iw}\to C_c^\infty(\tX,\Z)^{K_G^\Iw},\quad f\mapsto\Tr_{K_G^\Iw}^{K_G^\Iw\cap\tau^{-1}K_G^\Iw\tau}(\tau^{-1} f).
\]
Its action agrees with the operator $\mu_G(K_G^\Iw)^{-1}\indf[K_G^\Iw\tau^{-1}K_G^\Iw]$. In the notation of Subsection \ref{ss:IntCoeff}, with the choice $P=B_{\tG}$, we have $\tau\in \tG^+$ and $\sh{U}$ is supported on $\tG^-$.

The main result of this section is the following abstract wild norm relation.

\begin{prop}\label{prop:WildRelation}
  For $t\geq 0$,
  \[
    \Tr_{K_G^\Iw\times J_t}^{K_G^\Iw\times J_{t+1}}\delta^{(t+1)}=\sh{U}\cdot \delta^{(t)}.
  \]
\end{prop}
\begin{proof}
  Both sides can be computed using Lemma~\ref{lem:Trace}. For the left hand side,
  \begin{align*}
    \Tr_{K_G^\Iw\times J_t}^{K_G^\Iw\times J_{t+1}}\delta^{(t+1)}&=\indf[\mathbf{x}\tau^{t+1}(K_G^\Iw\times J_t)]\cdot[\Stab_{K_G^\Iw\times J_t}(\mathbf{x}\tau^{t+1}):\Stab_{K_G^\Iw\times J_{t+1}}(\mathbf{x}\tau^{t+1})]\\
    &=\indf[\mathbf{x}\tau^{t+1}(K_G^\Iw\times J_t)]\cdot[\Stab_{\tau^{t+1}K_G^\Iw\tau^{-(t+1)}\times J_{t}}(\mathbf{x}):\Stab_{\tau^{t+1}K_G^\Iw\tau^{-(t+1)}\times J_{t+1}}(\mathbf{x})].
  \end{align*}
  For the right hand side,
  \begin{align*}
    \sh{U}\cdot\delta^{(t)}&=\Tr_{K_G^\Iw\times J_t}^{(K_G^\Iw\cap\tau^{-1}K_G^\Iw\tau)\times J_t}\indf[\mathbf{x}\tau^t(K_G^\Iw\times J_t)\tau]\\
    &=\Tr_{K_G^\Iw\times J_t}^{(K_G^\Iw\cap\tau^{-1}K_G^\Iw\tau)\times J_t}\indf[\mathbf{x}\tau^{t+1}(\tau^{-1}K_G^\Iw\tau\times J_t)]\\
    &=\indf[\mathbf{x}\tau^{t+1}(K_G^\Iw\times J_t)]\cdot\frac{[\tau^{-1}K_G^\Iw\tau:K_G^\Iw\cap\tau^{-1}K_G^\Iw\tau]}{[\Stab_{\tau^{-1}K_G^\Iw\tau\times J_t}(\mathbf{x}\tau^{t+1}):\Stab_{K_G^\Iw\times J_t}(\mathbf{x}\tau^{t+1})]}\\
    &=\indf[\mathbf{x}\tau^{t+1}(K_G^\Iw\times J_t)]\cdot\frac{[\tau^{-1}K_G^\Iw\tau:K_G^\Iw\cap\tau^{-1}K_G^\Iw\tau]}{[\Stab_{\tau^{t}K_G^\Iw\tau^{-t}\times J_t}(\mathbf{x}):\Stab_{\tau^{t+1}K_G^\Iw\tau^{-(t+1)}\times J_t}(\mathbf{x})]}.
  \end{align*}
  We will now show that the two extraneous factors are both equal to 1. This is exactly \cite[Lemma 4.11]{LiuRS}. We include a more conceptual proof for completeness.
  
  Let $N$ (resp.\ $\overline{N}$) be the unipotent radical of $B_\tG$ (resp.\ $\overline{B}_\tG$), then for any $t\in\Z$ and $s\geq 0$, we have the Iwahori factorization
  \[
    \tau^t K_G^\Iw\tau^{-t}\times J_s=\big(\overline{B}_\tG\cap(\tau^t K_G^\Iw\tau^{-t}\times J_s)\big)\times \tau^t N(\Or)\tau^{-t}.
  \]
  Moreover, if $t\geq 0$ and $n\in\tau^t N(\Or)\tau^{-t}$, then $n\equiv 1\pmod{\varpi^t}$.
  
  Let $g\in\tau^t K_G^\Iw\tau^{-t}\times J_t$, and write $g=\bar{b}n$ with respect to this factorization. Suppose $\mathbf{x}g=\mathbf{x}$, then
  \[
    \mathbf{x}\bar{b}=\mathbf{x}n^{-1}\equiv\mathbf{x}\pmod{\varpi^t}.
  \]
  Since $\bar{b}\in\overline{B}_\tG$, we see that $\bar{b}$ is uniquely determined by $n^{-1}$. The final congruence implies that $\bar{b}\equiv 1\pmod{\varpi^t}$. In particular, the $F^\times$-factor of $\bar{b}$ lies in $J_t$. Therefore, the natural inclusion
  \[
    \Stab_{\tau^t K_G^\Iw\tau^{-t}\times J_t}(\mathbf{x})\subseteq\Stab_{\tau^t K_G^\Iw\tau^{-t}\times F^\times}(\mathbf{x})
  \]
  is actually an equality. It follows that the index term for the left hand side of the proposition is 1.
  
  When $t\geq 0$, the intersection $\overline{B}_\tG\cap(\tau^t K_G^\Iw\tau^{-t}\times J_s)$ contains $\overline{N}(\Or)$, so if $s\leq t$, then there is a bijection
  \[
    \tau^t N(\Or)\tau^{-t}\iso\Stab_{\tau^t K_G^\Iw\tau^{-t}\times J_s}(\mathbf{x})
  \]
  sending $n$ to $\bar{b}n$ for the unique $\bar{b}\in\overline{B}_\tG$ such that $\mathbf{x}\bar{b}=\mathbf{x}n^{-1}$. It follows that
  \begin{align*}
    [\Stab_{\tau^{t}K_G^\Iw\tau^{-t}\times J_t}(\mathbf{x}):\Stab_{\tau^{t+1}K_G^\Iw\tau^{-(t+1)}\times J_t}(\mathbf{x})]&=[\tau^t N(\Or)\tau^{-t}:\tau^{t+1}N(\Or)\tau^{-(t+1)}]\\
    &=[\tau^{-1}N(\Or)\tau:N(\Or)]\\
    &=[\tau^{-1}K_G^\Iw\tau:K_G^\Iw\cap\tau^{-1}K_G^\Iw\tau],
  \end{align*}
  where the final equality again follows from the Iwahori decomposition. This completes the proof.
\end{proof}

\begin{remark}
  This is essentially the same as the construction formulated in \cite{LoefflerSpherical}.
\end{remark}

\section{Tame norm relation}\label{sec:Tame}
This section will prove the tame norm relation. The proof is a technical local computation, and it is independent of the rest of the paper. One can skip to Section~\ref{ss:TameRelation}, which contains a statement of the main result in a form that immediately applies to the construction of Euler systems.

We will continue using the local notations introduced in Section~\ref{ss:LocalNotations}. Moreover, fix an additive character $\psi:F/\Or\to\C^\times$ such that $\psi(\varpi^{-1})\neq 1$. For psychological reasons, we will denote $\# k_F$ by $\ell$ instead of $q$.
\subsection{Whittaker models}
For each positive integer $n$, let $N_n(F)$ denote the subgroup of unipotent upper triangular matrices in $\GL_n(F)$, then $\psi$ extends to a character of $N_n(F)$ by
\[
  \psi(u)=\psi(u_{12}+\cdots+u_{n-1,n}).
\]
The space of Whittaker functions on $\GL_n(F)$ with character $\psi$ is defined by
\[
  \sh{W}_n(\psi)=\{W:\GL_n(F)\to\C\text{ smooth}\,|\,W(ug)=\psi(u)W(g)\text{ for all }u\in N_n(F),\ g\in\GL_n(F)\}.
\]
For $\underline{a}\in\Z^n$, define $\varpi^{\underline{a}}=\diag(\varpi^{a_1},\cdots,\varpi^{a_n})$. By the Iwasawa decomposition, $W\in\sh{W}_n(\psi)$ is uniquely determined by its values on matrices of the form $\varpi^{\underline{a}}k$ with $k\in\GL_n(\Or)$. The group $\GL_n(F)$ acts on $\sh{W}_n(\psi)$ by right translation, and we let $\sh{W}_n(\psi)^\circ\subseteq\sh{W}_n(\psi)$ be the subspace of (right) $\GL_n(\Or)$-invariant functions.

If $\pi$ is a smooth admissible representation of $\GL_n(F)$, then it is classical that
\[
  \dim_\C\Hom_{\GL_n(F)}(\pi,\sh{W}_n(\psi))\leq 1
\]
and the dimension is 1 for a dense set of irreducible unramified representations. In the dimension one case, let $\sh{W}_n(\pi,\psi)$ denote the image of any non-trivial map in the Hom-space. This is the Whittaker model of $\pi$.

\subsection{Local Birch lemma}\label{sec:Birch}
Let $\varphi$ denote the diagonal matrix $\varpi^{(-1,\cdots,-n)}$. Let
\[
  \sh{N}_n:=\varphi^{-1}N_n(\Or)\varphi/N_n(\Or).
\]
Given integers $j>i>0$, let $R_{ij}$ denote a set of representatives for $\varpi^{i-j}\Or/\Or$ in $F$, then
\[
  \mathscr{N}_n:=\{\eta\in N_n(F)\,|\,\eta_{ij}\in R_{ij} \text{ if }1\leq i<j\leq n\}
\]
is a set of coset representatives for $\sh{N}_n$. It has cardinality $\prod_{1\leq j\leq i\leq n}\ell^{j-i}=\ell^{\frac{1}{6}n(n+1)(n-1)}$. Let $s:F\to\{0,1\}$ be the indicator function of $\Or$. By an abuse of notation, also define $s:\sh{N}_n\to\N$ by $s(\eta)=\sum_{i=1}^{n-1} s(\eta_{i,i+1})$. This counts the number of integral off-diagonal elements of $\eta$.

\begin{lemma}\label{lem:Birch}
  Let $W_{n+1}\in\sh{W}_{n+1}(\psi)^\circ$, then the function 
  \[
    h\mapsto\widetilde{W_{n+1}}(h):=\sum_{\eta\in\sh{N}_{n+1}}(1-\ell)^{s(\eta)}W_{n+1}(\iota_{n,n+1}(h)\eta)
  \]
  is supported on $N_n(F)\GL_n(\Or)$. Moreover, if $k\in\GL_n(\Or)$, then
  \[
    \widetilde{W_{n+1}}(k)=\begin{cases}
      (-1)^n\ell^{\frac{1}{6}n(n+1)(n+2)}W_{n+1}(1) & \text{if }v(k_{ij})>i-j \text{ whenever }i>j\\
      0 & \text{otherwise}.
    \end{cases}
  \]
\end{lemma}
\begin{proof}
  First observe that $\widetilde{W_{n+1}}$ is well-defined since $W_{n+1}$ is spherical, so we can write the sum over $\mathscr{N}_{n+1}$ instead. Given $\eta\in\mathscr{N}_{n+1}$, it can be partitioned into
  \[
    \eta=\begin{pmatrix}
      \xi & v\\
       & 1
    \end{pmatrix},\quad \xi\in\mathscr{N}_n,\ v\in C_n:=R_n\times\cdots\times R_1.
  \]
  By performing the sum over $v\in C_n$ first, we get
  \[
    \widetilde{W_{n+1}}(h)=\sum_{\xi\in\mathscr{N}_n}(1-\ell)^{s(\xi)}\sum_{v\in C_n}(1-\ell)^{s(v_n)}W_{n+1}(\iota_{n,n+1}(h)\eta).
  \]
  It is clear that $\widetilde{W_{n+1}}\in\sh{W}_n(\psi)$, so $\widetilde{W_{n+1}}$ is determined by its values on matrices of the form $\varpi^{\underline{a}}k$, where $\underline{a}\in\Z^n$ and $k\in\GL_n(\Or)$. The inner sum is then
  \[
    (\dagger):=\sum_{v\in C_n}(1-\ell)^{s(v_n)}W_{n+1}\left(\begin{pmatrix} \varpi^{\underline{a}}k & \\ & 1 \end{pmatrix}\cdot\begin{pmatrix} \xi & v\\ & 1 \end{pmatrix}\right).
  \]

  Observe that
  \[
    \begin{pmatrix}\varpi^{\underline{a}}k & \\ & 1\end{pmatrix}\begin{pmatrix}\xi & v\\ & 1\end{pmatrix}=\begin{pmatrix}\mathrm{id}_n & \varpi^{\underline{a}}kv\\ & 1\end{pmatrix}\begin{pmatrix}\varpi^{\underline{a}}k\xi & \\ & 1\end{pmatrix}.
  \]
  Since $W_{n+1}\in\sh{W}_{n+1}( \psi)$, we get that
  \[
    W_{n+1}\left(\begin{pmatrix} \varpi^{\underline{a}}k & \\ & 1 \end{pmatrix}\cdot\begin{pmatrix} \xi & v\\ & 1 \end{pmatrix}\right)=\psi(\varpi^{a_n}(kv)_n)W_{n+1}\left(\begin{pmatrix}\varpi^{\underline{a}}k\xi & \\ & 1\end{pmatrix}\right)=\psi(\varpi^{a_n}(kv)_n)W_n(\varpi^{\underline{a}}k\xi),
  \]
  where $W_n:=W_{n+1}\circ\iota_{n,n+1}$. The first term can be expanded as
  \[
    \psi(\varpi^{a_n}(kv)_n)=\prod_{i=1}^n\psi(\varpi^{a_n}k_{ni}v_i).
  \]
  Therefore, 
  \[
    (\dagger)=\left(\prod_{i=1}^n\sum_{v_i\in R_{n-i+1}}(\ast)\,\psi(\varpi^{a_n}k_{ni}v_i)\right)W_n(\varpi^{\underline{a}}k\xi),
  \]
  where $(\ast)=1$ if $i\neq n$, and $(\ast)=(1-\ell)^{s(v_n)}$ if $i=n$.
  
  For an arbitrary $x\in\Or$ and index $i$, replacing $v_i$ by $v_i+x$ multiplies this product by $\psi(\varpi^{a_n}x)$. On the other hand, we have observed that this does not change the value of $(\dagger)$. Therefore, $(\dagger)=0$ unless $a_n\geq 0$. Assuming this, we get $\varpi^{a_n}k_{ni}\in\Or$ for all $i$, so standard orthogonality relations gives
  \[
    \sum_{v_i\in R_{n-i+1}}\psi(\varpi^{a_n}k_{ni}v_i)=\begin{cases}
      \ell^{n-i+1} & \text{if }v(k_{ni})+a_n\geq n-i+1\\
      0 & \text{otherwise}.
    \end{cases}
  \]
  For $i=n$, the inclusion of the factor $(1-\ell)^{s(v_n)}$ gives
  \[
    \sum_{v_n\in R_1}(1-\ell)^{s(v_n)}\psi(\varpi^{a_n}k_{nn}v_n)=\begin{cases}
      -\ell & \text{if }v(k_{nn})+a_n=0\\
      0 & \text{otherwise}.
    \end{cases}
  \]
  The expression $(\dagger)$ is a product of the above terms. For it to be non-zero, we must have $v(k_{nn})=a_n=0$ and $v(k_{ni})\geq n-i+1$ for all $i<n$. Therefore,
  \begin{equation}\label{eq:Birch-Inner}
    (\dagger)=\begin{cases}
      -\ell^{\frac{1}{2}n(n+1)}W_n(\varpi^{\underline{a}}k\xi) & \text{if }a_n=0\text{ and for all }i<n,\ v(k_{ni})>n-i\\
      0 & \text{otherwise}.
    \end{cases}
  \end{equation}
  
  Assuming the above conditions are met, we are left with
  \[
    \widetilde{W_{n+1}}(\varpi^{\underline{a_n}}k)=-\ell^{\frac{1}{2}n(n+1)}\sum_{\xi\in\mathscr{N}_n}(1-\ell)^{s(\xi)}W_n(\varpi^{\underline{a}}k\xi).
  \]
  Decompose $k$ as follows:
  \[
    k=\begin{pmatrix}\mathrm{id}_n & x\\ & 1\end{pmatrix}\cdot\begin{pmatrix}k' & \\ & 1\end{pmatrix}\cdot\begin{pmatrix}\mathrm{id}_n & \\y & u\end{pmatrix}=\begin{pmatrix}
      k'+xy & ux\\ y & u
    \end{pmatrix}.
  \]
  By comparing entries, we see that $u=k_{nn}$ is a unit, and for all $i\leq n-1$, $y_i=k_{ni}\in\varpi^{n-i+1}\Or$ and $x_i=u^{-1}k_{in}\in\Or$. Moreover,
  \begin{equation}\label{eqn:Birch-Vk}
    v((k')_{ij}-k_{ij})=v(x_i)+v(y_j)\geq n-j+1.
  \end{equation}
  We also have $a_n=0$, so
  \[
    W_n(\varpi^{\underline{a}}k\xi)=\psi(\varpi^{a_{n-1}}k_{n-1,n})W_n\left(\iota_{n-1,n}(\varpi^{\underline{a}'}k')\begin{pmatrix}\mathrm{id}_n & \\y & u\end{pmatrix}\xi\right),
  \]
  where $\underline{a}'=(a_1,\cdots,a_{n-1})$.
  
  Let $b=\big(\begin{smallmatrix}\mathrm{id}_n&\\y&u\end{smallmatrix}\big)$ and $B=\varphi b\varphi^{-1}$. The valuation bounds on $y$ implies that $B\in\GL_n(\Or)$ and $B\equiv\mathrm{id}_n\pmod{\varpi}$. Write $\xi=\varphi^{-1}\Xi\varphi$, then $\Xi\in N_n(\Or)$. By performing column operations on the matrix $B\Xi$, we obtain a factorization
  \[
    B\Xi=\Xi'C,
  \]
  where $\Xi'\in N_n(\Or)$ and $C$ a lower triangular matrix in $\GL_n(\Or)$. Moreover, $C\equiv \mathrm{id}_n\pmod{\varpi}$. Conjugating by $\varphi$ gives
  \[
    b\xi=\xi'(\varphi^{-1}C\varphi),\quad \xi':=\varphi^{-1}\Xi'\varphi\in\varphi^{-1}N_n(\Or)\varphi.
  \]
  Since $C$ is lower triangular, we still have $\varphi^{-1}C\varphi\in\GL_n(\Or)$. Therefore, this is the Iwasawa decomposition for $b\xi$, and $\xi\mapsto\xi'$ is a permutation of $\sh{N}_n$. By comparing the $(i,i+1)$-th entry of both sides, we see that
  \[
    \xi_{i,i+1}=(\xi')_{i,i+1}C_{i+1,i+1}+\sum_{j>i+1}(\xi')_{ij}C_{j,i+1}\varpi^{j-(i+1)}.
  \]
  Since $C\equiv \mathrm{id}_n\pmod{\varphi}$, each term in the sum is integral, and $C_{i+1,i+1}$ is a unit. Therefore, $\xi_{i,i+1}\in\Or$ if and only if $(\xi')_{i,i+1}\in\Or$. It follows that $s(\xi)=s(\xi')$.
  
  Combining everything together, we see that when the non-vanishing condition in equation~\eqref{eq:Birch-Inner} is satisfied,
  \begin{align*}
    \widetilde{W_{n+1}}(\varpi^{\underline{a}}k)&=-\ell^{\frac{1}{2}n(n+1)}\sum_{\xi\in\mathscr{N}_n}(1-\ell)^{s(\xi)}W_n(\varpi^{\underline{a}}k\xi)\\
    &=-\ell^{\frac{1}{2}n(n+1)}\psi(\varpi^{a_{n-1}}k_{n-1,n})\sum_{\xi\in\mathscr{N}_n}(1-\ell)^{s(\xi)}W_n\left(\iota_{n-1,n}(\varpi^{\underline{a}'}k')\begin{pmatrix}\mathrm{id}_n & \\y & u\end{pmatrix}\xi\right)\\
    &=-\ell^{\frac{1}{2}n(n+1)}\psi(\varpi^{a_{n-1}}k_{n-1,n})\sum_{\xi'\in\sh{N}_n}(1-\ell)^{s(\xi')}W_n(\iota_{n-1,n}(\varpi^{\underline{a}'}k')\xi').
  \end{align*}
  By induction, we may assume this lemma has been proven for $n$. Therefore, the sum is 0 unless $\underline{a}'=0$ and $v((k')_{ij})>i-j$ for all $1\leq j<i\leq n-1$. In particular, $a_{n-1}=0$, so the character term is trivial. Moreover, equation~\eqref{eqn:Birch-Vk} implies that
  \[
    v((k')_{ij})>i-j\iff v(k_{ij})>i-j.
  \]
  Combined with equation~\eqref{eq:Birch-Inner}, we obtain the lemma.
\end{proof}

\subsection{Hecke algebra and cyclicity}
Let $K_G=\GL_n(\Or)\times\GL_{n+1}(\Or)$ and $K_H=K_G\cap H=\GL_n(\Or)$. They are hyperspecial maximal compact subgroups of $G$ and $H$ respectively. In the rest of this section, fix the Haar measure on $G$ (resp.~$H$) so that $K_G$ (resp.~$K_H$) has volume 1.

The spherical Hecke algebra $\sh{H}_G$ is the set $C_c^\infty(K_G\backslash G/K_G,\C)$ of compactly supported $K_G$-bi-invariant functions on $G$ with multiplication given by convolution
\[
  (f_1\ast f_2)(g):=\int_G f_1(x)f_2(x^{-1}g)dx,
\]
For the augmented group $\widetilde{G}=G\times F^\times$, take the maximal compact subgroup $K_G\times\Or^\times$. Defining $\sh{H}_{\tG}$ analogously, then we have a relation
\[
  \sh{H}_{\widetilde{G}}=\sh{H}_G\otimes_\C C_c^\infty(F^\times/\Or^\times,\C)=\sh{H}_G[\mathtt{T}^{\pm 1}],
\]
where $\mathtt{T}$ is the indicator function of $\varpi\Or^\times$ in $F^\times/\Or^\times$. On both $\sh{H}_G$ and $\sh{H}_{\widetilde{G}}$, we can define the involution $f\mapsto f^\vee$ by $f^\vee(g):=f(g^{-1})$.

Given a smooth unramified representation $\widetilde{\pi}$ of $\widetilde{G}$, the space $\widetilde{\pi}^{K_{\widetilde{G}}}$ has the usual left $\sh{H}_{\widetilde{G}}$-action. The trace of $f\in\sh{H}_{\widetilde{G}}$ will be denoted by $\Tr\widetilde{\pi}(f)$. Of course, if $\widetilde{\pi}$ is irreducible, then $\widetilde{\pi}^{K_{\widetilde{G}}}$ is 1-dimensional, and $\widetilde{\pi}(f)v=\Tr\widetilde{\pi}(f)v$ for any $v\in\widetilde{\pi}$.

The space $C_c^\infty(\widetilde{X},\C)^{K_{\widetilde{G}}}$ has a natural left $\sh{H}_{\widetilde{G}}$-action, inherited from the right $\widetilde{G}$-action on $\widetilde{X}$. Let
\[
  \delta_0=\mathbf{1}[HK_{\widetilde{G}}]\in C_c^\infty\big(\widetilde{X},\C\big)^{K_{\widetilde{G}}}
\]
be the basic element. With our choice of the Haar measure, $I_H(\mathbf{1}[K_{\widetilde{G}}])=\delta_0$. The following proposition follows easily from some classical results.
\begin{prop}\label{prop:Cyc}
  The $\sh{H}_{\widetilde{G}}$-module $C_c^\infty(\widetilde{X},\C)^{\widetilde{G}}$ is generated by $\delta_0$.
\end{prop}
\begin{proof}
  This follows from \cite[Corollary 8.0.4(a)]{SakSpherical}, since the restriction map (8.2) at \emph{loc.~cit.}~is actually the identity in our setting. It is also an easy consequence of the calculations of \cite[Section 3]{MS_GLn}.
\end{proof}

Finally, we need a special Hecke operator whose trace gives the inverse of a local $L$-factor.
\begin{prop}\label{prop:HeckeL}
There exists a unique $\sh{L}\in\sh{H}_{\widetilde{G}}$ such that
  \[
    \Tr\widetilde{\pi}(\sh{L})=L\Big(\frac{1}{2},\widetilde{\pi}\Big)^{-1}
  \]
  for all generic irreducible unramified representations $\widetilde{\pi}$ of $\widetilde{G}$
\end{prop}
\begin{proof}
  We have the Satake isomorphism
  \[
    \sh{H}_\tG\simeq\C[A_1^{\pm 1},\cdots,A_n^{\pm 1}]^{S_n}\otimes_\C[B_1^{\pm 1},\cdots,B_{n+1}^{\pm 1}]^{S_{n+1}}\otimes_\C \C[T^{\pm 1}],
  \]
  where the symmetric groups $S_n$ and $S_{n+1}$ act by permuting their respective variables, and $\mathtt{T}$ is sent to $T$. Let $\sh{L}\in\sh{H}_\tG$ be the Hecke operator whose image is the polynomial
  \[
    \prod_{i=1}^n\prod_{j=1}^{n+1}\big(1-A_i B_j T\ell^{-\frac{1}{2}}\big).
  \]
  Comparing this with the definition of the local $L$-factor, we see that $\sh{L}$ satisfies the condition. Uniqueness follows since the set of Satake parameters for generic unramified representations is dense.
\end{proof}

\subsection{Zeta integral computation}
Let $K_H^\Iw$ be the Iwahori subgroup of $K_H$ of matrices which reduce to an upper triangular matrix mod $\varpi$. Recall from the previous subsection that $\varphi=\varpi^{(-1,\cdots,-n)}$. Define
\[
  K_H^\varphi=K_H^\Iw\cap\varphi K_H^\Iw\varphi^{-1}.
\]
This is the subgroup of $K_H$ defined by the valuation condition $v(k_{ij})>i-j$ whenever $i>j$. Let $\sh{N}_n$ and $s:\sh{N}_n\to\N$ be as in the Section~\ref{sec:Birch}. Define
\begin{equation}\label{eq:delta}
  \delta'=\sum_{\eta\in\sh{N}_n}\mu_H(K_H^\varphi)^{-1}(1-\ell)^{s(\eta)}\mathbf{1}[(1,\eta)K_G\times\Or^\times]\in C_c^\infty(\widetilde{G},\C)^{K_{G}\times\Or^\times}.
\end{equation}
By the general cyclicity result (Proposition~\ref{prop:Cyc}), we know that $I_H(\delta')=\sh{P}\cdot\delta_0$ for some $\sh{P}\in\sh{H}_{\widetilde{G}}$. We will now determine $\sh{P}$ using a zeta integral.
\begin{prop}\label{prop:TameZeta}
  Let $\sh{L}\in\sh{H}_{\widetilde{G}}$ be the element defined in Proposition~\ref{prop:HeckeL}, then
  \[
    (-1)^n\ell^{-\frac{1}{6}n(n+1)(n+2)}I_H(\delta')=\sh{L}^\vee\cdot\delta_0.
  \]
\end{prop}
\begin{proof}
  Let $\widetilde{\pi}$ be an irreducible generic unramified representation of $\widetilde{G}=\GL_n(F)\times\GL_{n+1}(F)\times F^\times$. With respect to this decomposition, it factors as $\widetilde{\pi}=\pi_n\otimes\pi_{n+1}\otimes\chi$. Fix a spherical vector $v^\circ\in\widetilde{\pi}^{K_{\widetilde{G}}}$. Given a map $\mathfrak{z}\in\Hom_H(\widetilde{\pi},\C)$, form the relative matrix coefficient
  \[
    \sh{S}:C_c^\infty({\widetilde{G}},\C)^{K_{\widetilde{G}}}\to\C,\ \phi\mapsto\mathfrak{z}(\widetilde{\pi}(\phi)v^\circ).
  \]
  Let $f\in\sh{H}_{\widetilde{G}}$ and $v\in\widetilde{\pi}^{K_{\widetilde{G}}}$, then
  \begin{align*}
    \widetilde{\pi}(f\cdot\phi)v&=\int_{\widetilde{G}} \int_{\widetilde{G}} \phi(xg)f(g)\widetilde{\pi}(x)v\,dgdx\\
    &=\int_{\widetilde{G}} \int_{\widetilde{G}} \phi(y)f(g)\widetilde{\pi}(yg^{-1})v\,dgdy\\
    &=\int_{\widetilde{G}}\phi(y)\pi(y)\int_{\widetilde{G}} f(g)\widetilde{\pi}(g^{-1})vdgdy\\
    &=\big(\Tr \pi(f^\vee)\big)\widetilde{\pi}(\phi)v.
  \end{align*}
  In particular, we get $\sh{S}(f\cdot\phi)=\Tr \widetilde{\pi}(f^\vee)\sh{S}(\phi)$.
  
  The map $\sh{S}$ is invariant under left translation by $H$, so it defines a quotient map
  \[
    \overline{\sh{S}}:C_c^\infty(\tX,\C)^{K_{\widetilde{G}}}\to\C,\quad\overline{\sh{S}}(I_H(\phi))=\sh{S}(\phi).
  \]
  Applying this to the relation $I_H(\delta')=\sh{P}\cdot\delta_0$ gives
  \[
    \mf{z}(\widetilde{\pi}(\delta')v^\circ)=\Tr\widetilde{\pi}(\sh{P}^\vee)\mf{z}(v^\circ).
  \]
  To prove this proposition, it remains to compute both sides for an explicit choice of $\mf{z}$ and $v^\circ$.
  
  Let $\sh{W}_{\widetilde{\pi}}:=\sh{W}_n(\pi_n,\psi^{-1})\otimes\sh{W}_{n+1}(\pi_{n+1},\psi)\otimes\chi$ denote the Whittaker model for $\widetilde{\pi}$. Note the difference in additive characters between the two terms. Define $\mf{z}\in\Hom_H(\widetilde{\pi},\C)$ by the following zeta integral
  \[
    \mf{z}\big(W_n\otimes W_{n+1}\big)=\int_{N_n\backslash H} W_n(h)W_{n+1}(\iota(h))\chi(\det h)dh.
  \]
  Let $W^\circ=W_n^\circ\otimes W_{n+1}^\circ$ be the spherical Whittaker function in $\sh{W}_\pi$, normalized so that $W_n^\circ(1)=W_{n+1}^\circ(1)=1$. The classical calculation of Jacquet--Shalika gives
  \[
    \mf{z}(W^\circ)=L\Big(\frac{1}{2},\widetilde{\pi}\Big).
  \]
  On the other hand, by Lemma~\ref{lem:Birch},
  \begin{align*}
    \mf{z}\big(\pi({\delta}')W^\circ\big)&=\sum_{\eta\in\sh{N}_n}(1-\ell)^{s(\eta)}\mu_H(K_H^\varphi)^{-1}\mf{z}((1,\eta)W^\circ)\\
    &=\int_{N_n\backslash H}\mu_H(K_H^\varphi)^{-1}W_n(h)\sum_{\eta\in\sh{N}_n}(1-\ell)^{s(\eta)}W_{n+1}(\iota(h)\eta)\chi(\det h)dh\\
    &=(-1)^n\ell^{\frac{1}{6}n(n+1)(n+2)}.
  \end{align*}
  The proposition follows.
\end{proof}

\subsection{Integrality}
The following groups were defined in Section~\ref{ss:LocalNotations}:
\[
  J_0=\Or^\times,\quad J_1=\{x\in\Or^\times\,|\,x\equiv 1\pmod\varpi\}.
\]
We have constructed an element $I_H(\delta')\in C_c^\infty(\widetilde{X},\C)^{K_G\times J_0}$. For Euler system applications, we need to know that it lands in the image of
\[
  C_c^\infty\big(\widetilde{X},\Z[\ell^{-1}]\big)^{K_G\times J_1}
\]
under the trace map $\Tr_{K_G\times J_0}^{K_G\times J_1}$. This will be done by another explicit computation.

Let $\mathscr{N}_n'\subseteq\mathscr{N}_n$ denote the set of $\eta$ such that $\varpi\eta_{i,i+1}\equiv 0\text{ or 1}\pmod\varpi$ for all $i$. Define
\begin{equation}\label{eqn:tame_twist}
  {\delta}_1'=\sum_{\eta\in\mathscr{N}'_n}\mu_H(K_H^\varphi)^{-1}(-1)^{s(\eta)}(\ell-1)^n\mathbf{1}[(1,\eta)K_G\times J_1]\in C_c^\infty\big(\widetilde{G},\C\big)^{K_G\times J_1}.
\end{equation}
\begin{prop}\label{prop:TameIntegral}
  $I_H(\delta'_1)\in C_c^\infty(\widetilde{X},\Z[\ell^{-1}])^{K_G\times J_1}$ and $\Tr_{K_G\times J_0}^{K_G\times J_1} I_H(\delta_1')=I_H(\delta')$.
\end{prop}
\begin{proof}
  We have
  \[
    I_H(\delta'_1)=\sum_{\eta\in\mathscr{N}'_n}(-1)^{s(\eta)}\frac{\mu_H(H_\eta)}{\mu_H(K_H^\varphi)}(\ell-1)^n\mathbf{1}[H((1,\eta)K_{G,1}\times J_1)],
  \]
  where $H_\eta$ is the stabilizer
  \[
    H_\eta:=H\cap\big((1,\eta)K_G(1,\eta^{-1})\times J_1\big).
  \]
  We first show that the coefficients are in $\Z[\ell^{-1}]$. Let
  \[
    K_{H,1}^\varphi=\{h\in K_H^\varphi\,|\,h_{ii}\equiv 1\ (\text{mod }\varpi)\text{ for all }i\}.
  \]
  Then $K_{H,1}^\varphi$ is a pro $\ell$-group, and $\mu_H(K_H^\varphi)=(\ell-1)^n\mu_H(K_{H,1}^\varphi)$. Therefore,
  \begin{align*}
    \frac{\mu_H(H_\eta)}{\mu_H(K_H^\varphi)}(\ell-1)^n &=\frac{\mu_H(H_\eta)}{\mu_H(K_{H,1}^\varphi)}\\
    &=\frac{\mu_H(H_\eta)}{\mu_H(H_\eta\cap K_{H,1}^\varphi)}\cdot \frac{\mu_H(H_\eta\cap K_{H,1}^\varphi)}{\mu_H(K_{H,1}^\varphi)}\\
    &=\frac{[H_\eta:H_\eta\cap K_{H,1}^\varphi]}{[K_{H,1}^\varphi:H_\eta\cap K_{H,1}^\varphi]}\in\Z[\ell^{-1}].
  \end{align*}
  
  By multiplying the columns and rows of an element $\eta\in\mathscr{N}_n$ by appropriate units, we can define a map $\mathtt{pr}:\mathscr{N}_n\to\mathscr{N}_n'$ such that $H(1,\eta)K_G=H(1,\mathtt{pr}(\eta))K_G$ for all $\eta$. The map $\mathtt{pr}$ also does not change the value of $s(\eta)$. The sizes of the fibres are $\#\mathtt{pr}^{-1}(\eta')=(\ell-1)^{n-s(\eta')}$, so we can decompose $\delta'$ as follows:
  \begin{align*}
    \delta'&=\sum_{\eta\in\mathscr{N}_n}\mu_H(K_H^\varphi)^{-1}(-1)^{s(\eta)}(\ell-1)^{s(\eta)}\mathbf{1}[(1,\eta)K_G\times J_0]\\
    &=\sum_{\eta'\in\mathscr{N}_n'}\mu_H(K_H^\varphi)^{-1}(-1)^{s(\eta')}(\ell-1)^n\cdot(\ell-1)^{-(n-s(\eta'))}\sum_{\mathtt{pr}(\eta)=\eta'}\mathbf{1}[(1,\eta)K_G\times J_0].
  \end{align*}
  Therefore,
  \[
  \begin{split}
    \Tr_{K_G\times J_0}^{K_G\times J_1}({\delta}'_1)-{\delta}'=\sum_{\eta'\in\mathscr{N}_n'}\mu_H(K_H^\varphi)^{-1} & (-1)^{s(\eta')}(\ell-1)^n \\
    & \times \left(\mathbf{1}[(1,\eta')K_G\times J_0]-\frac{1}{(\ell-1)^{n-s(\eta')}}\sum_{\mathtt{pr}(\eta)=\eta'}\mathbf{1}[(1,\eta)K_G\times J_0]\right).
  \end{split}
  \]
  After applying $I_H$, the term in the parenthesis becomes 0, giving the required result.
\end{proof}

\subsection{Summary}\label{ss:TameRelation}
We now summarize the above computations. Let $\sh{L}$ be the Hecke operator defined in Proposition~\ref{prop:HeckeL}, so $\Tr\pi(\sh{L})=L\big(\frac{1}{2},\pi\big)^{-1}$ for all unramified representations $\pi$ of $\tG$.
\begin{prop}\label{prop:TameRelation}
There exists an element
\[
  \delta\in C_c^\infty(\widetilde{X},\Z[\ell^{-1}])^{K_G\times J_1}
\]
such that 
\[
  \Tr_{K_G\times J_0}^{K_G\times J_1}\delta=\sh{L}^\vee\cdot\indf[H\cdot(K_G\times J_0)]\in C_c^\infty(\widetilde{X},\Z[\ell^{-1}])^{K_G\times J_0}.
\]
\end{prop}
\begin{proof}
Let $\delta=(-1)^n\ell^{-\frac{1}{6}n(n+1)(n+2)}I_H\big({\delta}'_1\big)$, where ${\delta}'_1$ was defined in equation~\eqref{eqn:tame_twist}. The proposition follows from Propositions~\ref{prop:TameIntegral} and \ref{prop:TameZeta}.
\end{proof}

\section{Construction of an Euler system}
We now use the formalism described in the previous sections to construct an anticyclotomic Euler system.

\subsection{Automorphic set-up}
Let $\pi$ be a cuspidal automorphic representation of $G(\A)$ of weight $\xi$, where $\xi$ is perfectly interlacing (Definition~\ref{def:Interlace}). Let $p$ be a rational prime. We can find an open compact subgroup $K=\prod_{\ell} K_\ell\subseteq G(\A_f)$ such that:
\begin{itemize}
  \item $\pi^K\neq 0$.
  \item Any subgroup of $K$ is neat (see Remark~\ref{remark:neat}).
  \item $K_\q$ has an Iwahori factorization for each place $\q|p$.
  \item The projector $\mathtt{t}$ considered in Subsection~\ref{ss:AJMap} exists.
\end{itemize}
Let $S$ be a finite set of places of $F$ containing all places above 2 and $p$, and such that away from $S$, both $E$ and $\pi$ are unramified, the local component of $K$ is a hyperspecial maximal compact subgroup, and the local component of the projector $\mathtt{t}$ is the identity.

We also recall the coefficient field notations introduced in Subsection~\ref{ss:IntCoeff}. Let $\Phi$ be a finite extension of $\Q_p$ such that $E$ splits in $\Phi$ and the representation $\pi_f$ and its associated Galois representation $\rho_\pi$ can both be defined over $\Phi$. Let $\Or$ be the ring of integers of $\Phi$, and let $t$ be the uniformizer used to construct integral representations. Define $\nu_\Phi$ to be the additive valuation on $\Phi$, normalized by $\nu_\Phi(t)=1$.

For classes above $p$ and the wild norm relations, we will need to make additional assumptions. Suppose $\p$ is a place of $F$ above $p$ such that the following conditions hold:
\begin{equation}\label{eqn:spl}
  \begin{cases}
  \text{$\p$ splits in $E$.}\\
  \text{$\pi$ is unramified at $\p$.}\\
  \text{$K_\p$ is the fixed Iwahori subgroup of $\GL_n(F_\p)\times\GL_{n+1}(F_\p)$ used in Section~\ref{sec:Wild}.}
  \end{cases}\tag{spl}
\end{equation}
In Section~\ref{sec:Wild} we defined the Hecke operator $\sh{U}_\p=\mu_G(K_\p)^{-1}[K_\p\tau^{-1} K_\p]$, where
\[
  \tau=\left(\diag(\varpi^n,\cdots,\varpi),\diag(\varpi^n,\cdots,\varpi,1)\right)\in G(F_\p)
\]
and $\varpi$ is a uniformizer of $F_\p$. This is independent of the choice of the Haar measure $\mu_G$. Let $\mu$ be the highest weight character of $\xi$. Then by the definition of the integral lattices,
\[
  \sh{U}_\p|_{\h^\bullet_\et(\Sh_G,\mathbb{V}_\xi)}=t^{\nu_\Phi(\mu(\tau))}\sh{U}_\p|_{\h^\bullet_\et(\Sh_G,\mathbb{L}_\xi)}.
\]
In other words, we have an automatic divisibility of $\sh{U}_\p$-eigenvalues. This also follows by observing that $V_\pi$ is crystalline at $\p$, so its Newton polygon lies above its Hodge polygon. The following definition takes this divisibility into account.
\begin{definition}\label{defn:ord}
  The representation $\pi$ is \emph{ordinary} at $\p$ if there exists a non-zero $v\in \pi_\p^K$ and $\lambda_\p\in\Or^\times$ such that $\sh{U}_\p\cdot v=\lambda_\p t^{\nu_\Phi(\mu(\tau))}v$.
\end{definition}

As expected, we will need to assume $\pi$ is ordinary at $\p$ in order to construct norm-compatible classes above $\p$ (Corollary~\ref{cor:ES}). However, we will not need it for the tame norm relations (Theorem~\ref{thm:ES}).

\subsection{Twisting elements}
Let $\mathscr{L}$ be the set of places of $F$ away from $S$ that split in $E$. For each $\ell\in\mathscr{L}$, fix a place $\lambda$ of $E$ above $\ell$. Finally, let $\mathscr{R}^p$ be the set of ideals of $F$ of the form
\[
  \m=\ell_1\cdots\ell_i,\text{ where }i\geq 0, \text{ and }\ell_1\cdots,\ell_i\in\mathscr{L}\text{ are distinct}.
\]
If \eqref{eqn:spl} holds for the place $\p$, then let $\mathscr{R}$ be the set of ideals of the form $\p^t\mathfrak{r}$, where $t\geq 0$ and $\mathfrak{r}\in\mathscr{R}^p$. Let 
$\mathscr{R}^{(p)}$ mean $\mathscr{R}$ if \eqref{eqn:spl} holds and $\mathscr{R}^p$ otherwise.

Given $\m\in\mathscr{R}^{(p)}$, let $E[\m]$ be the ring class field with conductor $\m$. Recall from the discussion of Subsection~\ref{ss:Extension} that $\spec E[\m]$ can be naturally identified with the Shimura variety $\Sh_1(J[\m])$, where $J[\m]$ is the open compact subgroup of $\U(\mathtt{V}_1)(\A_f)$ consisting of all elements which are 1 modulo $\m$.

For each $\ell\in\mathscr{L}$, the choice of the place $\lambda$ gives an identification
\[
  (H,G)\times_F F_\ell\simeq(\GL_n(F_\ell),\GL_n(F_\ell)\times\GL_{n+1}(F_\ell)).
\]
Let $\delta_\ell\in C_c^\infty(\widetilde{X}(F_\ell),\Z[\Norm\ell^{-1}])$ be the element $\delta$ defined in Subsection~\ref{ss:TameRelation}.\footnote{There is a change of notation: here $\ell$ is a place of $F$, but in Section~\ref{sec:Tame}, $\ell$ denotes the size of its residue field.} If \eqref{eqn:spl} holds, then for each integer $t\geq 0$, let $\delta_{\p^t}\in C_c^\infty(\widetilde{X}(F_\p),\Z)$ be the element $\delta^{(t)}$ defined in Subsection~\ref{ss:WildRelation}. For any place $\ell$, define the basic element $\delta_{\ell,0}$ to be the indicator function of $H(F_\ell)\cdot(K_\ell\times\Or_\ell^\times)\subseteq\widetilde{X}(F_\ell)$.
\begin{definition}
  Let $\m=\p^t\ell_1\cdots\ell_i\in\mathscr{R}^{(p)}$, then define
  \[
    \delta[\m]=\delta_{\p^t}\delta_{\ell_1}\cdots\delta_{\ell_i}\prod_{\ell\nmid\m}\delta_{\ell,0}\in C_c^\infty(\widetilde{X}(\A_f)^+,\Z_p).
  \]
\end{definition}
By construction, $\delta[\m]$ is invariant under right translation by $K\times J[\m]$. Moreover, the two norm relations Propositions~\ref{prop:WildRelation} and \ref{prop:TameRelation} can be summarized as follows.
\begin{prop}\label{prop:AbstractNR}
If $\m,\m\ell\in\mathscr{R}^{(p)}$, then
\[
  \Tr_{J[\m]}^{J[\m\ell]}\delta[\m\ell]=\begin{cases}
    \sh{U}_\p\cdot\delta[\m] & \ell=\p\\
    \sh{L}_\ell^\vee\cdot\delta[\m] & \ell\neq\p.
  \end{cases}
\]
\end{prop}

\begin{remark}\label{rmk:base-class}
While the elements $\delta[\m]$ in general depend on the choices of the primes $\lambda$ made above, the element $\delta[1]$ clearly does not. 
\end{remark}

\subsection{Euler system}
Fix a non-zero vector $v\in\pi_f^K$. If \eqref{eqn:spl} holds and $\pi$ is ordinary at $\p$, then we further assume that $v$ is an eigenvector for $\sh{U}_\p$ with eigenvalue $\lambda_\p t^{\nu_\Phi(\mu(\tau))}$, where $\lambda_\p\in\Or^\times$. Using this as the vector in equation~\eqref{eq:AJ}, we obtain an $\Or$-lattice $T_\pi\subseteq V_\pi$ and the Abel--Jacobi map
\[
  \mathrm{AJ}_{\mathtt{t},v}:C_c^\infty(\tX(\A_f)^+,\Z_p)^{K\times U}\to\h^1(E,T_\pi\otimes\Or[\Sh_1(U)]).
\]
Observe that by our choice of Iwahori subgroup at $\p$, the twisting elements $\delta[\m]$ are in the domain of $\mathrm{AJ}_{\mathtt{t},v}$. Take $U=J[\m]$ as before. Then by equation~\eqref{eq:Gal1} and Shapiro's lemma,
\begin{equation}\label{eq:Shapiro}
  \h^1(E,T_\pi\otimes\Or[\Sh_1(J[\m])])=\h^1(E[\m],T_\pi).
\end{equation}
Let $z_\m$ be the image of $\mathrm{AJ}_{\mathtt{t},v}(\delta[\m])$ under this identification. Our main theorem is then a direct consequence of Proposition~\ref{prop:AbstractNR} and the Hecke-equivariance property of $\mathrm{AJ}_{\mathtt{t},v}$.
\begin{theorem}\label{thm:ES}
  If $\m,\m\ell\in\mathscr{R}^{(p)}$, then
  \[
    \Tr_{E[\m]}^{E[\m\ell]}z_{\m\ell}=\begin{cases}
      \lambda_\p z_\m & \text{if }\ell=\p\\
      P_\lambda(\Fr_\lambda^{-1})z_\m & \text{if }\ell\neq\p,
    \end{cases}
  \]
  where $P_\lambda(X)$ is the polynomial such that for all $s\in\C$,
  \begin{equation}\label{eq:CharPoly}
    P_\lambda(\Norm\ell^{-s})=L\Big(s+\frac{1}{2},\pi_\ell^\vee\Big)^{-1}
  \end{equation}
  and $\Fr_\lambda$ is the arithmetic Frobenius.
\end{theorem}
\begin{proof}
  Under the identification~\eqref{eq:Shapiro}, the trace map from $J[\m\ell]$ to $J[\m]$ is identified with the trace map from $E[\m\ell]$ to $E[\m]$. This is a consequence of the definition of the Galois action on $\Sh_1$. It follows that
  \[
    \Tr_{E[\m]}^{E[\m\ell]}z_{\m\ell}=\Tr_{J[\m]}^{J[\m\ell]}\mathrm{AJ}_{\mathtt{t},v}(\delta[\m\ell])=\mathrm{AJ}_{\mathtt{t},v}\big(\Tr_{J[\m]}^{J[\m\ell]}\delta[\m\ell]\big),
  \]
  where the final equality follows from the equivariance of $\mathrm{AJ}_{\mathtt{t},v}$ under Hecke operators on $\U(\mathtt{V}_1)$.
  
  The traces on the right hand side are given by Proposition~\ref{prop:AbstractNR}. First suppose $\ell=\p$. The element $\tau^{-1}$ lies in $\tG(\A_f)^-$, so we have the equivariance property
  \[
    \mathrm{AJ}_{\mathtt{t},v}\big(\Tr_{J[\m]}^{J[\m\ell]}\delta[\m\ell]\big)=\mathrm{AJ}_{\mathtt{t},v}\big(\sh{U}_\p\cdot\delta[\m]\big)=\lambda_\p\cdot\mathrm{AJ}_{\mathtt{t},v}(\delta[\m]),
  \]
  where the final step follows from the choice of the vector $v$.
  
  If $\ell\neq\p$, then we need to consider the spherical Hecke operator $\sh{L}^\vee_\ell$. Using the notations in Proposition~\ref{prop:HeckeL}, the Satake transform of $\sh{L}^\vee_\ell$ is
  \[
    \prod_{i=1}^n\prod_{j=1}^{n+1}\big(1-A_i^{-1}B_j^{-1}T^{-1}\Norm\ell^{-\frac{1}{2}}\big).
  \]
  The inverses are due to the presence of the involution $(-)^\vee$. We have a decomposition
  \[
    \sh{H}_{\tG(F_\ell)}=\sh{H}_{G(F_\ell)}\otimes_\C C_c^\infty(\Or_\ell^\times\backslash F_\ell^\times,\C)\simeq\sh{H}_{G(F_\ell)}\otimes_\C \C[\mathtt{T}^{\pm 1}],
  \]
  where $\mathtt{T}$ is the indicator function of $\varpi\Or_\ell^\times$ and gets sent to $T$ under the Satake isomorphism.
  
  By the Shimura reciprocity law, the Hecke action of $\mathtt{T}^{-i}$ is identified with $\Fr_\lambda^{-i}$. Let $\alpha_1,\cdots,\alpha_n$ (resp.~$\beta_1,\cdots,\beta_{n+1}$) be the Satake parameters of $\pi_{n,\ell}$ (resp.~$\pi_{n+1,\ell}$). Then
  \[
    \mathrm{AJ}_{\mathtt{t},v}(\sh{L}_\ell^\vee\cdot\delta[\m])=\prod_{i=1}^n\prod_{j=1}^{n+1}\big(1-\alpha_i^{-1}\beta_j^{-1}\Norm\ell^{-\frac{1}{2}}\Fr_\lambda^{-1})\cdot\mathrm{AJ}_{\mathtt{t},v}(\delta[\m]).
  \]
  Comparing with the definition of the local $L$-factor, the left hand side is exactly $P_\lambda(\Fr_\lambda^{-1})$.
\end{proof}

\begin{cor}\label{cor:ES}
  Suppose in addition that $\pi$ is ordinary at $\p$. Then there exists a collection of class
  \[
    \big\{c_\m\in \h^1(E[\m],T_\pi)\,|\,\m\in\mathscr{R}\}
  \]
  satisfying the norm relations
  \[
    \Tr_{E[\m]}^{E[\m\ell]}c_{\m\ell}=\begin{cases}
      c_\m & \text{if }\ell=\p\\
      P_\lambda(\Fr_\lambda)c_\m & \text{if }\ell\neq\p.
    \end{cases}
  \]
\end{cor}
\begin{proof}
  Since (ord) holds for $\pi$, we may further assume $\lambda_\p$ is a unit in the previous theorem. For $\m\in\mathscr{R}$, define $c_\m=\lambda_\p^{-t}z_\m$, where $t$ is the power of $\p$ appearing in the factorization of $\m$. They lie in the same lattice since $\lambda_\p\in\Or^\times$, and the norm relations are immediate from the theorem.
\end{proof}

\begin{remark}
  In both Theorem~\ref{thm:ES} and Corollary~\ref{cor:ES}, the base class $c_1=z_1$ is obtained by applying an Abel--Jacobi map to the diagonal cycle studied in the arithmetic Gan--Gross--Prasad conjecture. It is expected to be generically non-trivial, and its vanishing is conjecturally related to the central derivative of an $L$-function. Note that the classes $c_1= z_1$ do not depend on the choices made for the definition of $\mathscr{R}^p$.
\end{remark}

\subsection{Applications}
Let $\Pi=\Pi_n\times\Pi_{n+1}$ be a RACSDC automorphic form on $\GL_n(\A_E)\times\GL_{n+1}(\A_E)$. Suppose the archimedean component of $\Pi$ has the same infinitesimal character as a perfectly interlacing algebraic representation $\xi$. By the self-duality, the root number $\varepsilon\big(\frac{1}{2},\Pi\big)$ is $\pm 1$.
\begin{assumption}
  \(
    \varepsilon\big(\frac{1}{2},\Pi\big)=-1.
  \)
\end{assumption}
Following the discussion of \cite[Chapter 27]{GGPConjecture} (where the conjectural Arthur multiplicity formula is established in our special case by \cite{KMSW}), there exist
Hermitian spaces $\mathtt{V}_n\subseteq\mathtt{V}_{n+1}$ and a cuspidal automorphic form $\pi$ on $G(\A_F)$ of the type we have been considering such that $\Pi$ is the base change of $\pi$. The local components of $\mathtt{V}_n,\mathtt{V}_{n+1}$,and $\pi$ are specified by the local conjecture \cite[Chapter 17]{GGPConjecture}. In particular, the perfect interlacing condition on weights forces the archimedean signatures of $\mathtt{V}_n$ and $\mathtt{V}_{n+1}$ to be standard indefinite (Subsection~\ref{ss:Herm}).

Assuming Hypothesis~\ref{conj:Coh}, the Galois representation $T_\Pi$ constructed from the cohomology of certain unitary Shimura varieties is a lattice in the Galois representation $V_\Pi$ attached to $\Pi$, in the sense that
\[
  L(s,V_\Pi)=L\Big(s+\frac{1}{2},\Pi\Big).
\]
From this, it is easy to verify that
\[
  P_\lambda(X)=\det(1-X\Fr_\lambda|V_\Pi)^{-1},
\]
where $P_\lambda$ is the polynomial from the norm relations defined by~\eqref{eq:CharPoly}. The presence of the contragredient there matches with our choice of using the arithmetic Frobenius here. It follows that our norm relation in Corollary~\ref{cor:ES} is indeed the split anti-cyclotomic Euler system norm relation for the Galois representation $V_\Pi$. By applying the results of \cite{JNS}, we obtain the following theorem.

\begin{theorem}\label{thm:App} Let $z_E = \Tr_E^{E[1]} z_1 \in \h^1_f(E,V_\Pi)$.  Suppose
\begin{itemize} 
\item[\rm (i)] $V_\Pi$ is absolutely irreducible,
\item[\rm (ii)] there exists $\sigma \in \Gal_E$ that fixes $E[1](\mu_{p^\infty})$ and such that $\dim_\Phi V_\Pi/(\sigma-1)V_\Pi = 1$, 
\item[\rm (iii)] either 
\subitem{\rm (a)} $\mathscr{R}^{(p)} = \mathscr{R}$ and $\Pi$ is ordinary at $\mathfrak{p}$, or
\subitem{\rm (b)} $\mathscr{R}^{(p)} = \mathscr{R}^p$ and there exists $\gamma \in\Gal_E$ such that $\gamma$ fixes $E[1](\mu_{p^\infty})$
and $V_\Pi/(\gamma-1)V_\Pi = 0$.
\end{itemize}
Then
\begin{equation}\label{rk1BK}
z_E\neq 0 \implies \dim_\Phi\h^1_f(E,V_\Pi) = 1.
\end{equation}
\end{theorem}

As an example, we will now give one set of purely automorphic conditions which imply the conditions of Theorem~\ref{thm:App} for all except possibly two primes $p$.

\begin{cor}\label{cor:App}
  Let $\Pi=\Pi_n\times\Pi_{n+1}$ be as before. Suppose there exists finite places $v_n,v_{n+1}$ of $E$ such that
  \begin{itemize}
    \item $v_n$ and $v_{n+1}$ lie above different places in $F$;
    \item for $m=n,n+1$, $\Pi_m$ is a twist of a Steinberg representation at $v_m$;
    \item $\Pi_n$ is unramified at $v_{n+1}$, and $\Pi_{n+1}$ is unramified at $v_n$.
  \end{itemize}
  Then the implication \eqref{rk1BK} holds for all $p$ not dividing the residue characteristics of $v_n$ and $v_{n+1}$.
\end{cor}
\begin{proof}
  We verify hypotheses (i), (ii), and (iii)(b) of Theorem~\ref{thm:App}.
  \begin{itemize}
      \item[(i)] Let $\{m,m'\} = \{n,n+1\}$. Let $I_{v_m}$ be the inertia subgroup at $v_m$. Local-global compatibility at $v_m$ \cite{Caraiani} implies that there exists an element $\tau_m \in I_{v_m}$ such that $\rho_{\Pi_m}(\tau_m)$ is a unipotent element with only one Jordan block. Moreover, $V_{\Pi_{m'}}$ is unramified at $v_{m}$, so $\rho_{\Pi_{m'}}(\tau_m)=1$.  It follows that $\sigma = \tau_n\tau_{n+1}\in \Gal_E$ is such that $\rho_\Pi(\sigma)$ is unipotent with only one Jordan block.  This implies that $\rho_{\Pi}$ is absolutely irreducible.
      \item[(ii)] The $\sigma$ constructed above satisfies $\dim_\Phi V_\Pi/(\sigma-1)V_\Pi=1$. It is the product of elements in the inertia subgroups $I_{v_n}$ and $I_{v_{n+1}}$, so it fixes $E[1]$. Since $p$ does not divide the residue characteristics of $v_n$ and $v_{n+1}$, the element $\sigma$ also fixes $\mu_{p^\infty}$.
      \item[(iii)(b)] Let $G \subset \Gal_{E[1](\mu_{p^\infty})}$ be the subgroup such that $\det\rho(G)=1$.  Note that $\sigma\in G$. It then follows from a classification theorem of Katz \cite[Proposition 4]{Scholl} that the connected component of the Zariski closure of $\rho(G)$ is one of 
$\mathrm{Sym}^{d-1}\SL_2$, $\mathrm{SL}_d$, or $\mathrm{Sp}_d$, where $d=n(n+1)$ is even. Each of these groups contains a regular semisimple element with no eigenvalue equal to $1$
(for $\mathrm{Sym}^{d-1}\SL_2$ we are using that $d-1$ is odd). Hence so must $\rho(G)$. Such an element gives the desired $\gamma$.\qedhere
  \end{itemize}
\end{proof}

\begin{remark}
  If we assume further that for each $m\in\{n,n+1\}$, there are two places $v_m$ and $v_m'$ of distinct residue characteristics such that 
  $\Pi_m$ is a twist of a Steinberg representation at both primes, and $\Pi_n$ is unramified at $v_{n+1}, v_{n+1}'$, and $\Pi_{n+1}$ is unramified at $v_n, v_n'$, then by applying the corollary to the four possible pairs, we see that the implication \eqref{rk1BK} holds for all primes $p$.
\end{remark}

\bibliographystyle{alpha}
\bibliography{Ref}

\end{document}